\documentclass[11pt]{amsart}
\newtheorem{theorem}{Theorem}[section]

\newtheorem{corollary}[theorem]{Corollary}

\newtheorem{lemma}[theorem]{Lemma}

\newtheorem{problem}[theorem]{Problem}
\newtheorem{proposition}[theorem]{Proposition}
\newtheorem{remark}[theorem]{Remark}

\usepackage[active]{srcltx} %SRC Specials for DVI search
\def\pair#1#2{\langle#1,#2\rangle}
\def\ip#1#2{(#1|#2)}
\def\E{{\widehat{E}}}

\def\J#1#2#3{ \left\{ #1,#2,#3 \right\} }
\def\Jc#1#2#3{ \left\{ #1,#2,#3 \right\} }

\def\NN{{\mathbb{N}}}
\def\11{\textbf{$1$}}
\def\CC{{\mathbb{C}}}

\usepackage{enumerate}
\usepackage{amsmath}
\usepackage{amssymb}

\begin{document}

\title{Ternary Weakly Amenable C$^*$-algebras and JB$^*$-triples}

\author[T. Ho]{Tony Ho}
\email{zh\underline{\ }ho01@yahoo.com}
\address{13416 Sheridan Ave,
Urbandale, IA  50323}

\author[A.M. Peralta]{Antonio M. Peralta}
\email{aperalta@ugr.es}
\address{Departamento de An{\'a}lisis Matem{\'a}tico, Facultad de
Ciencias, Universidad de Granada, 18071 Granada, Spain}

\author[B. Russo]{Bernard Russo}
\email{brusso@uci.edu}
\address{Department of Mathematics, University of California, Irvine CA, USA}

\thanks{Second author partially supported by
D.G.I. project no. MTM2008-02186, and Junta de Andaluc\'{\i}a
grants FQM0199 and FQM3737.}

\date{}

\begin{abstract}
A well known result of  Haagerup from 1983 states that every C$^*$-algebra, $A$, is
weakly amenable, that is, every (associative) derivation from $A$ into its dual is inner.
A Banach algebra, $B$, is said to be ternary weakly amenable if every continuous Jordan triple derivation
from $B$ into its dual is inner. We show that commutative C$^*$-algebras are ternary weakly amenable,
but that $B(H)$ and $K(H)$ are not, unless $H$ is finite dimensional. More generally, we inaugurate
the study of weak amenability for Jordan Banach triples, focussing on commutative JB$^*$-triples
and some Cartan factors.
\end{abstract}

\maketitle
 \thispagestyle{empty}

%\tableofcontents

\section{Introduction}

Two fundamental questions concerning derivations from a Banach algebra $A$ into a Banach $A$-bimodule $M$ are:
\begin{itemize}
\item Is an everywhere defined derivation automatically continuous?
\item Are all continuous derivations inner?
If not, can every continuous derivation be approximated by inner derivations?
\end{itemize}

One can ask the same questions in the setting of Jordan Banach algebras (and Jordan modules),
and more generally for Jordan Banach triple systems (and Jordan Banach triple modules).
Significant special cases occur in each context when $M=A$ or when $M=A^*$.\smallskip

In order to obtain a better perspective on the objectives of this paper, we shall give here a comprehensive review of the major existing results on these two problems in the contexts in which we will be interested, namely, C$^*$-algebras, JB$^*$-algebras, and JB$^*$-triples.
Although we will be dealing with both the real and complex cases in this paper, in the interest of space this review will be confined to the complex case.
\smallskip

A derivation on a Banach algebra, $A,$ into a Banach $A$-bimodule, $M,$ is a linear mapping
$D:A\rightarrow M$ such that $D(ab)=a\cdot D(b) + D(a)\cdot b$. An inner derivation, in this context,
is a derivation of the form:  $\hbox{ad}_x (a)=x\cdot a-a\cdot x$ ($x\in M$, $a\in A$).\smallskip

In the context of C$^*$-algebras, automatic continuity results were
initiated by Kaplansky before 1950 (see \cite{Kaplansky58}) and culminated in the following series of results:
Every derivation from a C$^*$-algebra into itself is continuous (Sakai, 1960 \cite{Sak60}); Every derivation from a
C$^*$-algebra $A$  into a Banach $A$-bimodule is continuous (Ringrose, 1972 \cite{Ringrose72}).\smallskip

The major results for C$^*$-algebras regarding inner derivations read as follows: Every derivation from a
C$^*$-algebra on a Hilbert space $H$ into itself is of the form $x\mapsto ax-xa$ for some $a$ in the weak closure of the
C$^*$-algebra in $L(H)$ (Sakai \cite{sakai}, Kadison \cite{Kad}, 1966); Every amenable
C$^*$-algebra is nuclear (Connes, 1976 \cite{connes}); Every nuclear C$^*$-algebra is amenable (Haagerup, 1983 \cite{Haa83});
Every C$^*$-algebra is weakly amenable (Haagerup, 1983 \cite{Haa83} and Haagerup-Laustsen, 1998 \cite{HaaLaust}).
For finite dimensional C$^*$-algebras, the last result follows from the work of Hochschild  in 1942 \cite{Hochs}.
 \smallskip

As a bridge to the Jordan algebra setting, we make a slight digression.
Sinclair proved in 1970 (cf. \cite{Sinclair70}) that a continuous Jordan
derivation from a semisimple Banach algebra to itself is a
derivation, although this result fails for derivations of
semisimple Banach algebras into a Banach bi-module.  (A {\it Jordan derivation} from a Banach algebra $A$ into a Banach
$A$-module is a linear map $D$ satisfying $D(a^2) = a D(a) + D(a)
a,$ ($a\in A$), or equivalently, $D(ab+ba)=aD(b)+D(b)a + D(a)b
+bD(a),$ ($a,b\in A$).)
Nevertheless, Johnson proved in 1996 (cf. \cite{John96}) that every continuous Jordan derivation from a
C$^*$-algebra $A$ to a Banach $A$-bimodule is a
derivation.  A new proof of this fact was
presented by U. Haagerup and N.J. Laustsen in \cite{HaaLaust}.\smallskip

The following  subsequent result partially  removed the assumption of continuity from this theorem of Johnson: Every Jordan derivation from a
von Neumann algebra, or from a commutative C$^*$-algebra, into a Banach bimodule  is continuous (Alaminos-Bre\v{s}ar-Villena 2004 \cite{AlBreVill}).
More recently, the assumption was completely removed: Every Jordan derivation from an arbitrary
C$^*$-algebra  into a Banach bimodule  is continuous (Peralta-Russo, 2010 \cite[Cor. 22]{PeRu}).
Earlier, Cusack, completing a study of Sinclair, showed that every Jordan derivation on a semisimple Banach algebra is continuous \cite{Cusack75},
and Villena extended this result to semisimple Jordan Banach algebras \cite{Villena96}.\smallskip

We now move to the context of Jordan Banach algebras. A derivation from a Jordan Banach algebra
$A$ into a Jordan Banach module $M$ is a linear mapping $D:A\rightarrow M$ such that $D(a\circ b)=a\circ Db+Da\circ b$,
where $\circ$ denotes both the product in the Jordan algebra and the module action.
(Jordan Banach algebras and Jordan Banach modules will be defined below.) An inner derivation in this context is a
derivation of the form:  $\sum_{i=1}^{m} \left(L(x_i)L(a_i)-L(a_i)L(x_i)\right)$, ($x_i\in M, a_i\in A$).
Here, $L(x)$ is the operator $a\mapsto a\circ x$ from $A$ to $M$ and $L(a)$ is either the operator $b\mapsto b\circ a$
from $A$ to $A$ or $x\mapsto a\circ x$ from $M$ to $M$. \smallskip

In the context of JB$^*$-algebras, the major automatic continuity results consist of the following.
Every (Jordan) derivation of a reversible JC$^*$-algebra extends to a derivation (associative)
of its enveloping C$^*$-algebra (Upmeier, 1980 \cite{upmeier}---this recovers
Sinclair's result in the case of C$^*$-algebras);
Every Jordan derivation from a JB$^*$-algebra $A$ into $A$ or into $A^*$ is continuous and
every Jordan derivation from a commutative or a compact C$^*$-algebra into a Jordan Banach bimodule,
is continuous (Hejazian-Niknam, 1996 \cite{HejNik96}).
This latter result was also extended to arbitrary C$^*$-algebras in \cite[Cor. 21]{PeRu}:
Every Jordan derivation from an arbitrary C$^*$-algebra $A$ into a Jordan Banach $A$-bimodule is continuous.\smallskip

The major results for JB$^*$-algebras regarding inner derivations are the following: Every Jordan derivation from a finite dimensional JB$^*$-algebra
into a Jordan Banach module is inner (follows from Jacobson, 1951 \cite{Jacobson51}, \cite{Jac}); Every Jordan derivation of a purely exceptional or a reversible JBW-algebra is inner (Upmeier, 1980 \cite{upmeier}); Every Jordan derivation of $\oplus L^\infty(S_j,U_j)$ ($U_j$ spin factors) is inner if and only if $\sup_j\dim U_j<\infty$, \cite{upmeier}. By a structure theorem for JBW-algebras, these theorems of Upmeier completely determine whether a given JBW-algebra has only inner derivations.\smallskip

Finally, we move to a discussion of Jordan Banach triples, which is the proper setting for this paper. A (triple or ternary) derivation on a Jordan Banach triple $A$ into a Jordan Banach triple module $M$ is a {\it conjugate} linear mapping $D:A\rightarrow M$ such that  $D\{a,b,c\}=\{Da,b,c\}+\{a,Db,c\}+\{a,b,Dc\}$. An
inner derivation in this context is a derivation of the form:   $\sum_i^{m} \left(L(x_i,a_i)-L(a_i,x_i)\right)$, ($x_i\in M, a_i\in A$), where $L(x,a)$ and $L(a,x)$ denote the maps $b\mapsto \J{x}{a}{b}$ and $b\mapsto\J{a}{x}{b}$ arising from the module action. (Jordan Banach triple and Jordan Banach triple module will be defined below, after which the reason for the conjugate linearity in the complex case of derivations into a module, as opposed to linearity,  will be more transparent.)\smallskip

In the context of JB$^*$-triples, automatic continuity results were initiated by Barton and Friedman in 1990 (cf.\  \cite{BarFri}) who showed that every triple derivation of a JB$^*$-triple is continuous.  Peralta and Russo in 2010 (see \cite[Theorem 13]{PeRu}) gave necessary and sufficient conditions under which a derivation of a JB$^*$-triple into a Jordan Banach triple module is continuous. As shown in \cite{PeRu}, these conditions are automatically satisfied in the case that the JB$^*$-triple is actually a C$^*$-algebra with the triple product $(xy^*z+zy^*x)/2$, leading to a new proof (cf. \cite[Cor. 23]{PeRu}) of the theorem of Ringrose quoted above as well as the results of Alaminos-Bre\v{s}ar-Villena and Hejazian-Niknam, also quoted above.\smallskip

The known results for JB$^*$-triples regarding inner derivations are surveyed in the following statements: Every derivation from a finite dimensional JB$^*$-triple into itself is inner (follows from Meyberg, 1972 \cite{Meyberg72}); Every derivation from a finite dimensional JB$^*$-triple into a Jordan Banach triple module is inner (follows from K\"uhn-Rosendahl, 1978 \cite{KR78}); Every derivation of a Cartan factor of type $I_{n,n}$ ($n$ finite or infinite), type II (with underlying Hilbert space of even or infinite dimension) or type III is inner (Ho-Martinez-Peralta-Russo, 2002 \cite{HoMarPeRu}); Infinite dimensional Cartan factors of type $I_{m,n}, m\ne n$, and type IV have derivations into themselves which are not inner (cf. \cite{HoMarPeRu}).\smallskip

It is worth noting that, besides the consequences for C$^*$-algebras of the main result of \cite{PeRu} noted above, another consequence is the automatic continuity of derivations of a JB$^*$-triple into its dual \cite[Cor. 15]{PeRu}, leading us to the study of weak amenability for JB$^*$-triples,
which is the main focus of this paper. \smallskip

We conclude this review introduction by describing the contents of this paper. Section 2 sets down the definitions and basic properties of Jordan triples, Jordan triple modules, derivations and (ternary) weak amenability that we shall use.  Sections 3 and 4 are concerned  with C$^*$-algebras, considered as JB$^*$-triples with the triple product $(xy^*z+zy^*x)/2$. It is proved that commutative C$^*$-algebras are ternary weakly amenable, and that the compact operators, as well as all bounded operators on a Hilbert space $H$ are ternary weakly amenable if and only if $H$ is finite dimensional.\smallskip

Sections 5 and 6 are concerned with more general JB$^*$-triples.  It is proved that certain Cartan factors (Hilbert spaces and spin factors) are ternary weakly amenable if and only if they are finite dimensional, that infinite dimensional finite rank Cartan factors of type 1 are not ternary weakly amenable,  and that commutative JB$^*$-triples are almost weakly amenable in the sense that the inner derivations into the dual are norm dense in the set of all derivations into the dual.  By comparison, the existing forerunners on approximation of derivations on C$^*$-algebras by inner derivations (immediate consequence of the Sakai-Kadison results (\cite{Kad},\cite{sakai}), JB$^*$-algebras (\cite{upmeier}), and JB$^*$-triples (\cite{BarFri})  involved the topology of pointwise convergence and not the norm topology.

The authors wish to acknowledge the helpful comments and suggestions made by the referee.

\section{Derivations on Jordan triples; Jordan triple modules}

\subsection{Jordan triples}

A complex (resp., real) \emph{Jordan triple} is a complex (resp., real)
vector space $E$ equipped with a non-trivial  triple
product $$ E \times E \times E \rightarrow E$$
$$(xyz) \mapsto \J xyz $$
which is bilinear and symmetric in the outer variables and
conjugate linear (resp., linear) in the middle one satisfying the
so-called \emph{``Jordan Identity''}:
$$L(a,b) L(x,y) -  L(x,y) L(a,b) = L(L(a,b)x,y) - L(x,L(b,a)y),$$
for all $a,b,x,y$ in $E$, where $L(x,y) z := \J xyz$. When $E$ is
a normed space and the triple product of $E$ is continuous, we say
that $E$ is a \emph{normed Jordan triple}. If a normed Jordan
triple $E$  is complete with respect to the norm (i.e. if $E$ is a
Banach space), then it is called a
\emph{Jordan Banach triple}. Every normed Jordan triple can be
completed in the usual way to become a Jordan Banach triple.
Unless otherwise specified, the term ``normed Jordan triple''
(resp., ``Jordan Banach triple'') will always mean a real or
complex normed Jordan triple (resp., a real or complex
Jordan Banach triple).\smallskip

%For each Jordan-Banach triple $E$, the constant $N(E)$ will denote
%the supremum of the set $\{\|\J xyz \|: \|x\|,\|y\|,\|z\|\leq
%1\}$.\smallskip

A subspace $F$ of a Jordan triple $E$ is said to be a
\emph{subtriple} if $\J FFF \subseteq F$. We recall that a
subspace $J$ of $E$ is said to be a \emph{triple ideal} if
$\{E,E,J\}+\{E,J,E\} \subseteq J.$ When $\J JEJ \subseteq J$ we say
that $J$ is an \emph{inner ideal} of $E$.\smallskip

A real (resp., complex) \emph{Jordan algebra} is a
(non-necessarily associative) algebra over the real (resp.,
complex) field whose product is abelian and satisfies $(a \circ
b)\circ a^2 = a\circ (b \circ a^2)$. A normed Jordan algebra is a
Jordan algebra $A$ equipped with a norm, $\|.\|$, satisfying $\|
a\circ b\| \leq \|a\| \ \|b\|$, $a,b\in A$. A \emph{Jordan Banach
algebra} is a normed Jordan algebra whose norm is
complete.\smallskip

A Jordan algebra is called {\it special} if it is isomorphic to a subspace of an associative
algebra which is closed under $ab+ba$.
Every Jordan algebra is a Jordan triple with respect to
$$\J abc := (a\circ b) \circ c + (c\circ b) \circ a - (a\circ c) \circ b.$$
%Such a Jordan triple is called a {\it  $j$-special} Jordan triple.

If a  Jordan triple arises from a special Jordan algebra, then the
triple product reduces to $\J abc = \frac12 (a bc +cba)$.
Thus, every real or complex associative Banach algebra (resp., Jordan
Banach algebra) is a real Jordan-Banach triple with respect to the
product $\J abc = \frac12 (a bc +cba)$ (resp., $\J abc = (a\circ
b) \circ c + (c\circ b) \circ a - (a\circ c) \circ b$).\smallskip

A real or complex Jordan-Banach triple $E$ is said to be \emph{commutative}
or \emph{abelian} if the identity $$ \J {\J xyz}ab = \J xy{\J zab} = \J x{\J yza}b$$ holds for all
$x,y,z,a,b \in E$, equivalently, $L(a,b) L(c,d) = L(c,d) L(a,b)$, for every $a,b,c,d\in E$.\smallskip

A JB$^*$-algebra is a complex Jordan Banach algebra $A$ equipped
with an algebra involution $^*$ satisfying  $\|\J a{a^*}a \|= \|a\|^3$, $a\in
A$.  (Recall that $\J a{a^*}a  =
 2 (a\circ a^*) \circ a - a^2 \circ a^*$.)\smallskip

A \emph{(complex) JB$^*$-triple} is a complex Jordan Banach triple
${E}$ satisfying the following axioms: \begin{enumerate}[($JB^*
1$)] \item For each $a$ in ${E}$ the map $L(a,a)$ is an hermitian
operator on $E$ with non negative spectrum; \item  $\left\|
\{a,a,a\}\right\| =\left\| a\right\| ^3$ for all $a$ in ${A}.$
\end{enumerate}\smallskip

Every C$^*$-algebra (resp., every JB$^*$-algebra) is a JB$^*$-triple with respect to the product
$\J abc = \frac12 \ ( a b^* c + cb^* a) $ (resp., $\J abc := (a\circ b^*) \circ c + (c\circ b^*) \circ a - (a\circ c) \circ b^*$).\smallskip

A summary of the basic facts about JB$^*$-triples, an important and well-understood
class of Jordan Banach triples, some of which are recalled here, can be
found in \cite{Russo94} and some of the references therein, such
as \cite{Ka}, \cite{FriRu85}, \cite{FriRus86bis}, \cite{U1} and \cite{U2}.\smallskip

We recall that a \emph{real JB$^*$-triple} is a norm-closed real
subtriple of a complex JB$^*$-triple (compare
\cite{IsKaRo95}). The class of real JB$^*$-triples includes all complex JB$^*$-triples,
all real and complex C$^*$- and JB$^*$-algebras and all JB-algebras.\smallskip

A complex (resp., real) \emph{JBW$^*$-triple} is a complex (resp.,
real) JB$^*$-triple which is also a dual Banach space (with a
unique isometric predual \cite{BarTi,MarPe}). It is known that the
triple product of a real or complex JBW$^*$-triple is separately weak$^*$
continuous (cf. \cite{BarTi} and \cite{MarPe}). The second dual
of a JB$^*$-triple $E$ is a JBW$^*$-triple with a product
extending the product of $E$ \cite{Di86b,IsKaRo95}.\smallskip

JB-algebras are precisely the self adjoint parts
of JB$^*$-algebras \cite{Wright77}, and a JBW-algebra is a JB-algebra which is
a dual space.\smallskip

When $E$ is a (complex) JB$^*$-triple or a real JB$^*$-triple, a
subtriple $I$ of $E$ is a triple ideal if and only if $\J EEI
\subseteq I$ or $\J EIE \subseteq I$
 or $\J EII\subseteq I$ (compare \cite[Proposition 1.3]{BuChu}).\smallskip

\subsection{Jordan triple modules}

Let $A$ be an associative algebra. Let us recall that an
\emph{$A$-bimodule} is a vector space $X$, equipped with two
bilinear products $(a,x)\mapsto a x$ and $(a,x)\mapsto x a$ from
$A\times X$ to $X$ satisfying the following axioms: $$a (b x) = (a
b) x ,\ \ a (x b) = (a x) b, \hbox{ and, }(xa) b = x (a b),$$ for
every $a,b\in A$ and $x\in X$.\smallskip

Let $A$ be a Jordan algebra. A \emph{Jordan $A$-module} is a
vector space $X$, equipped with two bilinear products
$(a,x)\mapsto a \circ x$ and $(x,a)\mapsto x \circ a$ from
$A\times X$ to $X$, satisfying: $$a \circ x = x\circ a,\ \ a^2
\circ (x \circ a) = (a^2\circ  x)\circ a, \hbox{ and, }$$ $$2(
(x\circ a)\circ  b) \circ a + x\circ (a^2 \circ b) = 2 (x\circ
a)\circ  (a\circ b) + (x\circ b)\circ a^2,$$ for every $a,b\in A$
and $x\in X$. The space $A\oplus X$ is a Jordan algebra with respect
to the product $$(a,x)\circ (b,y) := (a\circ b, a \circ y + b\circ x).$$
The Jordan algebra $(A\oplus X, \circ)$ is called the Jordan split null
extension of $A$ and $X$ (compare \cite[\S II.5,p.82]{Jac}).
When $A$ is a Jordan Banach algebra, $X$ is
a Banach space and the mapping $A\times X \to X$, $(a,x) \mapsto a\circ x$ is continuous,
then $X$ is said to be a Jordan Banach module.  The Jordan split null
extension is never a $JB$-algebra since $(0,x)^2=0$. \smallskip

Let $E$ be a complex (resp. real) Jordan triple.  A \emph{Jordan
triple $E$-module}  (also called \emph{triple $E$-module}) is a
vector space $X$ equipped with three mappings $$\{.,.,.\}_1 :
X\times E\times E \to X, \quad \{.,.,.\}_2 : E\times X\times E \to
X$$ $$ \hbox{ and } \{.,.,.\}_3: E\times E\times X \to X$$
satisfying  the following axioms:
\begin{enumerate}[{$(JTM1)$}]
\item $\{ x,a,b \}_1$ is linear in $a$ and $x$ and conjugate
linear in $b$ (resp., trilinear), $\{ abx \}_3$ is linear in $b$
and $x$ and conjugate linear in $a$ (resp., trilinear) and
$\{a,x,b\}_2$ is conjugate linear in $a,b,x$ (resp., trilinear)
\item  $\{ x,b,a \}_1 = \{ a,b,x \}_3$, and $\{ a,x,b \}_2  = \{
b,x,a \}_2$  for every $a,b\in E$ and $x\in X$. \item Denoting by
$\J ...$ any of the products $\{ .,.,. \}_1$, $\{ .,.,. \}_2$ or
$\{ .,.,. \}_3$, the identity $\J {a}{b}{\J cde} = \J{\J abc}de $
$- \J c{\J bad}e +\J cd{\J abe},$ holds whenever one of the
elements  $a,b,c,d,e$ is in $X$ and the rest are in $E$.
\end{enumerate}

It is obvious that every real or complex Jordan triple $E$ is a
{\it real}  triple $E$-module.  It is problematical whether
every complex Jordan triple $E$ is a complex triple $E$-module for
a suitable triple product.  This is partly why we have defined (in section 1 and subsection 2.3) a derivation of a complex
JB$^*$-triple into a Jordan Banach triple module to be conjugate linear.\smallskip

When $E$ is a Jordan Banach triple and $X$ is a triple $E$-module
which is  also a Banach space and, for each $a,b$ in $E$, the
mappings $x\mapsto \J abx_3$ and $x\mapsto \J axb_2$ are
continuous, we shall say that $X$ is a triple $E$-module with
\emph{continuous module operations}. When the products $\{ .,.,.
\}_1$, $\{ .,.,. \}_2$ and $\{ .,.,. \}_3$ are (jointly)
continuous we shall say that $X$ is a \emph{Banach (Jordan) triple
$E$-module}. \smallskip

Hereafter, the triple products $\J \cdot\cdot\cdot_j$, $ j=1,2,3$,
will be simply denoted by $\J \cdot\cdot\cdot$ whenever the
meaning is clear from the context.
\smallskip

Every (associative) Banach $A$-bimodule (resp., Jordan Banach $A$-module)
$X$ over an associative Banach algebra $A$ (resp., Jordan Banach algebra $A$)
is a real Banach triple $A$-module (resp., $A$-module) with respect to the ``\emph{elementary}''
product $$\J abc := \frac12 ( a b c+cba)$$ (resp., $\J abc= (a\circ b) \circ c + (c\circ b) \circ a
- (a\circ c) \circ b$), where one element of $a,b,c$ is in $X$ and the other two are in $A$.\smallskip

It is easy but  laborious to check that the dual space,
$E^*$, of a complex (resp., real) Jordan Banach triple $E$ is  a
complex (resp., real) triple $E$-module with respect to the
products: \begin{equation}\label{eq module product dual 1}   \J
ab{\varphi} (x) = \J {\varphi}ba (x) := \varphi \J bax
\end{equation} and \begin{equation}\label{eq module product dual
2} \J a{\varphi}b (x) := \overline{ \varphi \J axb }, \forall x\in
X, a,b\in E. \end{equation}

Given a triple $E$-module $X$  over a Jordan triple $E$, the
space $E\oplus X$ can be equipped with a structure of real Jordan
triple with respect to the product $\J {a_1+x_1}{a_2+x_2}{a_3+x_3}
= \J {a_1}{a_2}{a_3} +\J  {x_1}{a_2}{a_3}+\J  {a_1}{x_2}{a_3} + \J
{a_1}{a_2}{x_3}$. Consistent with the terminology in \cite[\S
II.5]{Jac}, $E\oplus X$ will be called the \emph{triple split null extension}
of $E$ and $X$. It is never a $JB^*$-triple. \smallskip

A subspace $S$ of a triple $E$-module $X$ is said to be a
\emph{Jordan triple submodule} or a \emph{triple submodule} if and
only if $\J EES + \J ESE \subseteq S$. Every triple ideal $J$ of
$E$ is a Jordan triple $E$-submodule of $E$.\smallskip

\subsection{Derivations}

Let $X$ be a Banach $A$-bimodule over an (associative) Banach algebra $A$. A linear mapping $D : A \to X$ is said to be
a \emph{(binary or associative) derivation} if $D(a b) = D(a) b + a D(b)$, for every $a,b$ in $A$. The symbol $\mathcal{D}_b(A,X)$
will denote the set of all continuous binary derivations from $A$ to $X$.
\smallskip

When $X$ is a Jordan Banach module over a Jordan Banach algebra $A$, a linear mapping $D : A \to X$ is said to be
a \emph{Jordan derivation} if $D(a \circ b) = D(a) \circ b + a \circ D(b)$, for every $a,b$ in $A$.
We denote the set of continuous Jordan derivations from $A$ to $X$ by  $\mathcal{D}_J(A,X)$. Although
Jordan derivations also are binary derivations, we use the word ``binary'' only for associative derivations.
\smallskip

%When $J$ is a Jordan algebra, it can be easily deduced, from the Jordan identity, that the mapping $\delta_{a,b} : J \to J,$
%$c\mapsto (a\circ c)\circ b - (b\circ c)\circ a$, is a Jordan derivation, for every $a,b\in J$.
%An {\it inner derivation} of $J$ is a finite sum of derivations  of the form $\delta_{a,b}$
%with $a,b\in J$.\smallskip

%For completeness, we include a sketch of the proof of the above statement.
%In the Jordan algebra axiom $$u(u^2v)=u^2(uv),$$ replace $u$ by $u+w$ to obtain the equation
%\begin{equation}\label{eq:1.4}
%2u((uw)v)+w(u^2v)=2(uw)(uv)+u^2(wv).\end{equation}
%In (\ref{eq:1.4}), interchange $v$ and $w$ and subtract the resulting equation from  (\ref{eq:1.4})
%to obtain the equation
%\begin{equation}\label{eq:1.5}
%2u(\delta_{v,w}(u))=\delta_{v,w}(u^2).
%\end{equation}
 % In (\ref{eq:1.5}), replace $u$ by $x+y$ to obtain  the desired result.
%\smallskip

In the setting of Jordan Banach triples, a \emph{triple} or \emph{ternary derivation} from a (real or complex)
Jordan Banach triple, $E,$ into a Banach triple $E$-module, $X$,
is a conjugate linear mapping $\delta: E \to X$ satisfying \begin{equation}\label{eq triple derivation} \delta \J abc =
\J {\delta (a)}bc +  \J a{\delta (b)}c + \J ab{\delta (c)},\end{equation} for every $a,b,c$ in $E$. The set of all continuous
ternary derivations from $E$ to $X$ will be denoted by $\mathcal{D}_t(E,X)$. According to \cite{BarFri} and \cite{HoMarPeRu},
a \emph{ternary derivation} on $E$ is a linear mapping $\delta : E\to E$ satisfying the identity $(\ref{eq triple derivation})$.
It should be remarked here that, unlike derivations from $E$ to itself, derivations from $E$ to $E^*$, when the latter is regarded as
a Jordan triple $E$-module, are defined to be conjugate linear maps (in the complex case). The words Jordan or ternary may seem
redundant in the expressions ``Jordan derivation on a Jordan algebra'' and ``ternary derivation on a
Jordan triple'', nevertheless,  we shall make use of them for clarity.\smallskip

If $E$ is a real or complex Jordan Banach triple,
we can easily conclude, from the Jordan identity,
that $\delta (a,b) := L(a,b)-L(b,a)$ is a
ternary derivation on $E$, for every $a,b\in E$. A {triple or ternary derivation}
$\delta$ on $E$ is said to be {\it inner} if it can be written as a finite sum of
derivations of the form $\delta(a,b)$ $(a,b\in E)$. Following \cite{BarFri} and \cite{HoMarPeRu},
we shall say that $E$ has the \emph{inner derivation property} if
every ternary derivation on $E$ is inner. The just quoted papers study the inner derivation property
in the setting of real and complex JB$^*$-triples.\smallskip

The following technical result will be needed later.

\begin{proposition}\label{p bidual extensions}%prop 2.3
 Let $E$ be a real or complex JB$^*$-triple
and let $\delta : E \to E^*$ be a ternary derivation. Then $\delta^{**} : E^{**} \to E^{***}$
is a weak$^*$-continuous ternary derivation satisfying $\delta^{**} (E^{**}) \subseteq E^{*}$.
\end{proposition}

\begin{proof} Let $\delta$ be a ternary derivation from a real (or complex) JB$^*$-triple to its dual, which is  automatically bounded by
\cite[Cor. 15]{PeRu}.
It is known that every bounded linear operator
from a real JB$^*$-triple to the dual of another real JB$^*$-triple factors through a real Hilbert space (cf. \cite[Lemma 5]{PeRo}).
Thus $\delta$ factors though a real Hilbert space and hence it is
weakly compact. By  \cite[Lemma 2.13.1]{HiPhi}, we have $\delta^{**} ({E}^{**}) \subset {E}^{*}$.\smallskip

We shall prove now that $\delta^{**}$ is a ternary derivation. Clearly, the mapping $\delta^{**} : E^{**} \to E^{***}$ is $\sigma(E^{**},E^*)$-to-$\sigma(E^{***},E^{**})$-continuous.
Let $a,b$ and $c$ be elements in $E^{**}$. By Goldstine's Theorem,
there exist (bounded) nets $(a_\lambda)$, $(b_\mu)$ and $(c_\beta)$ in $E$ such that $(a_\lambda)\to a$, $(b_\mu)\to b$ and $(c_\beta)\to c$
in the weak$^*$-topology of $E^{**}$.\smallskip

It should be noticed here that for every net $(\phi_{\lambda})$ in $E^{***}$,
converging to some $\phi\in E^{***}$ in the $\sigma (E^{***},E^{**})$-topology,
the nets $(\J {\phi_{\lambda}}ab)$ and $(\J a{\phi_{\lambda}}b)$ converge in the
$\sigma (E^{***},E^{**})$-topology to $(\J {\phi}ab)$ and $(\J a{\phi}b)$, respectively. Having this fact in mind,
it follows from the separate weak$^*$-continuity of the triple product in $E^{**}$ together with the weak$^*$-continuity of $\delta^{**}$ that $$\delta^{**} \J a{b_{\mu}}{c_{\beta}} = w^*\hbox{-}\lim_{\lambda} \delta \J {a_{\lambda}}{b_{\mu}}{c_{\beta}}$$
$$ = w^*\hbox{-}\lim_{\lambda} \J { \delta(a_{\lambda})}{b_{\mu}}{c_{\beta}}+ \J { a_{\lambda}}{\delta(b_{\mu})}{c_{\beta}} + \J { a_{\lambda}}{b_{\mu}}{\delta(c_{\beta})},$$
$$w^*\hbox{-}\lim_{\lambda} \J { \delta(a_{\lambda})}{b_{\mu}}{c_{\beta}} = \J { \delta^{**}(a)}{b_{\mu}}{c_{\delta}},$$ $$ w^*\hbox{-}\lim_{\lambda} \J { a_{\lambda}}{\delta(b_{\mu})}{c_{\beta}} = \J {a}{\delta(b_{\mu})}{c_{\beta}},$$ and
$$w^*\hbox{-}\lim_{\lambda} \J { a_{\lambda}}{b_{\mu}}{\delta(c_{\beta})} = \J { a}{b_{\mu}}{\delta(c_{\beta})},$$ for every $\mu$ and $\beta$. Therefore \begin{equation}\label{eq 1 prop biduals} \delta^{**} \J a{b_{\mu}}{c_{\beta}} = \J { \delta^{**}(a)}{b_{\mu}}{c_{\beta}} + \J {a}{\delta(b_{\mu})}{c_{\beta}}+\J { a}{b_{\mu}}{\delta(c_{\beta})},
\end{equation}  for every $\mu$ and $\beta$. By a similar argument, taking weak$^*$-limits in $(\ref{eq 1 prop biduals})$ first in $\mu$ and later in $\beta$, we get $$ \delta^{**} \J a{b}{c} = \J { \delta^{**}(a)}{b}{c} + \J {a}{\delta^{**}(b)}{c}+\J { a}{b}{\delta^{**}(c)},$$ which concludes the proof.
\end{proof}

\subsection{Weakly amenable Jordan Banach triples}

Let $X$ be a Banach $A$-bimodule over an associative Banach algebra $A$.
Given $x_{_0}$ in $X$, the mapping $D_{x_{_0}} : A \to X$,
$D_{x_{_0}} (a) = x_{_0} a - a x_{_0}$ is a bounded
(associative or binary) derivation. Derivations of this form are called
\emph{inner}. The set of all inner derivations from $A$ to $X$
will be denoted by $\mathcal{I}nn_{b} (A,X).$\smallskip

Recall that a Banach algebra $A$ is said to be \emph{amenable} if every bounded
derivation of $A$ into a dual $A$-module is inner, and  \emph{weakly amenable} if
every (bounded) derivation from $A$ to $A^*$ is inner. In
\cite{Haa83}, U. Haagerup, making use of preliminary work of J.W.
Bunce and W.L. Paschke \cite{BunPasch} and the
Pisier-Haagerup Grothendieck's inequality for general
C$^*$-algebras, proved that every C$^*$-algebra is weakly
amenable. In \cite{HaaLaust}, U. Haagerup and N.J. Laustsen gave a
simplified proof of this result.\smallskip

When $x_{_0}$ is an element in a Jordan Banach $A$-module, $X,$
over a Jordan Banach algebra, $A$, for each $b\in A$, the mapping
$\delta_{x_{_0},b} : A \to X$, $$\delta_{x_{_0},b} (a) := (x_{_0}
\circ a)\circ b - (b \circ a)\circ x_{_0}, \ (a\in A),$$ is a
bounded derivation. Finite sums of derivations of this form are
called \emph{inner}. The symbol $\mathcal{I}nn_{J} (A,X)$ will
stand for the set of all inner Jordan derivations from $A$ to
$X$.\smallskip

The Jordan Banach algebra $A$ is said to be \emph{weakly amenable},
or \emph{Jordan weakly amenable}, if every (bounded) derivation
from $A$ to $A^*$ is inner. It is
natural to ask whether every JB$^*$-algebra is weakly amenable.
The answer is no, as Lemma~\ref{lem:4.11} or Lemma~\ref{lem:4.11 L(H)} shows.\smallskip

In the more general setting of Jordan Banach triples the corresponding
definitions read as follows: Let $E$ be a complex (resp., real)
Jordan triple and let $X$ be a triple $E$-module. For each $b\in E$ and each $x_{_0}\in X,$
we conclude, via the main identity for Jordan triple modules $(JTM3)$, that
the mapping $\delta=\delta(b,x_{_0}) : E  \to X,$ defined by
\begin{equation}\label{eq:0308111}
\delta (a)=\delta(b,x_{_0}) (a):=\J{b}{x_{_0}}{a}-\J{x_{_0}}{b}{a} \ \ (a\in E),
\end{equation}
is a ternary derivation from $E$ into $X$.
Finite sums of derivations of the form $\delta(b,x_{_0})$ are
called \emph{(ternary) inner derivations}. Henceforth, we shall write
$\mathcal{I}nn_{t} (E,X)$ for the set of all inner ternary
derivations from $E$ to $X$. %The {\it degree} of an inner
%derivation $\delta\in  \mathcal{I}nn_{t} (E,X)$ is the least
%number of terms in a representation of $\delta$ as a finite sum of
%derivations of the form $\delta(b,x_{_0})$, with $b\in E$, $x_0\in X$.
\smallskip

A Jordan Banach triple $E$ is said to be \emph{weakly amenable} or
\emph{ternary weakly amenable} if every continuous triple derivation from
$E$ into its dual space is necessarily inner.  \smallskip

%\begin{remark}\label{r complexification}{\rm

In the next step we explore the connections between
ternary weak amenability\hyphenation{amena-bility} in a real JB$^*$-triple
and its complexification. Let $E$ be a real JB$^*$-triple.
By \cite[Proposition 2.2]{IsKaRo95}, there exists a unique complex JB$^*$-triple
structure on the complexification $\widehat {E} = E\oplus i\ E$,
and a unique conjugation (i.e., conjugate-linear isometry of
period 2) $\tau$ on $\widehat {E}$ such that $E=\widehat {E}
^{\tau} := \{x\in \widehat {E} : \tau (x)=x \}$, that is, $E$ is a
real form of a complex JB$^*$-triple. Let us consider $$
\tau^\sharp : \E^{*} \rightarrow \E^{*}$$ defined by
$$\tau^\sharp (\phi) (z) = \overline{\phi (\tau (z))}.$$
The mapping $\tau^\sharp$ is a conjugation on $\E^{*}$.
Furthermore the map $$(\E^{*})^{\tau^\sharp}
\longrightarrow (\E^{\tau})^{*}\ (=E^*)$$ $$\phi \mapsto \phi|_{E}$$ is an
isometric bijection, where $(\E^{*})^{\tau^\sharp} := \{ \phi \in
\E^{*} : \tau^\sharp (\phi)= \phi \},$ and thus $\widehat{E}^{*} = E^{*} \oplus i E^*$ (compare \cite[Page
316]{IsKaRo95}).\smallskip

Our next result is a module version of \cite[Proposition 1]{HoMarPeRu}.
We shall only include a sketch of the proof.

\begin{proposition}\label{p complexification}
A real JB$^*$-triple is ternary weakly amenable if and only if
its complexification has the same property.
\end{proposition}

\begin{proof} Let $E$ be a real JB$^*$-triple, whose complexification
is denoted by $\widehat {E} = E\oplus i\ E$, and let $\tau$ denote the conjugation
on $\widehat{E}$ satisfying $E=\widehat {E}^{\tau}$.\smallskip

According to \cite[Remark 13]{PeRu}, given a triple derivation
$\delta : E \to E^*$, the mapping $\widehat{\delta} : \widehat {E}
\to \widehat{E}^*$, $\widehat{\delta} (x+i y) := \delta (x) - i
\delta (y)$ is a (conjugate-linear) triple derivation from
$\widehat {E}$ into $\widehat{E}^*$. It can be easily checked that
the identity
\begin{equation}\label{eq degree 1 derivation complexification}
\delta (a+i b, \phi_1 +i \phi_2) = \delta (a, \phi_1)-
\delta (b, \phi_2) -i \delta (a, \phi_2) -i \delta (b, \phi_1),
\end{equation} holds for every $a,b\in E\subseteq \E$ and
$\phi_1,\phi_2\in E^*\subseteq \E^*$.\smallskip

Having in mind the facts proved in the above paragraph,
we can see that $E$ is ternary weakly amenable whenever
$\widehat{E}$ satisfies the same property.\smallskip

For the reciprocal implication, we notice that if $\widehat{\delta} : \widehat {E}
\to \widehat{E}^*$ is a ternary derivation from $\widehat {E} = E\otimes i E$ to
$\widehat{E}^* = E^* \otimes i E^*,$ it can be easily checked that the identity
$$(P\circ \widehat{\delta}) \J {a}{b}{c} =  \J {(P\circ \widehat{\delta})a}{b}{c}+
\J {a}{(P\circ \widehat{\delta})b}{c} + \J {a}{b}{(P\circ \widehat{\delta})c},$$
holds for every $a, b, c\in E = \widehat{E}^{\tau}$,
where $P$ denotes $\frac{Id_{E^*} + \tau^{\sharp}}{2}$ or $\frac{Id_{E^*} - \tau^{\sharp}}{2 i}$.
Therefore the mapping  $P \circ \widehat{\delta}|_{E}: E \to E^*$ is a ternary derivation.
Since every ternary inner derivation $\delta$ from $E$ to $E^*$
defines a ternary inner derivation from $\widehat {E}$ to
$\widehat{E}^*,$ we can guarantee that $\widehat {E}$ is ternary weakly amenable
when $E$ has this property.
\end{proof}

Note that if $A$ is a Banach $^*$-algebra, $A$ is ternary weakly amenable if each continuous
ternary derivation from $A$ (considered as a Jordan Banach triple system) into $A^*$ is inner.
Thus a Banach $^*$-algebra can be weakly amenable and/or ternary weakly amenable,
and the two concepts do not necessarily coincide (compare Proposition \ref{p compact operators is not ternary weakly amen}).\smallskip

We emphasize again that, unlike derivations from $A$ to itself, derivations from $A$ to $A^*$
are defined to be conjugate linear maps (in the complex case).\smallskip

\section{Commutative C$^*$-algebras are ternary weakly amenable}

In this section we prove that every commutative (real or complex)  C$^*$-algebra is
ternary weakly amenable. Our next results establish some technical
connections between associative and ternary derivations from a
Banach $^*$-algebra $A$ to a Jordan $A$-module (resp., associative
$A$-bimodule).\smallskip

Following standard notation, given a Banach algebra $A$,
$a\in A$ and $\varphi\in A^*$,
$a \varphi,~ \varphi a$ will denote the elements in
$A^*$ given by $$a\varphi(y) = \varphi(y a)\quad
\text{and} \quad \varphi a (y) = \varphi(a y), \ \ (y\in A).$$

\begin{lemma}\label{lem:4.1}%4.1
\label{l ternary and associative derivations} Let $A$ be an
associative unital (Banach) *-algebra (which we consider as a Jordan Banach algebra), $X$ a unital Jordan
$A$-module and let $\delta : A_{sa} \to X$ be a (real) linear
mapping. The following assertions are
equivalent:\begin{enumerate}[$(a)$] \item $\delta$ is a ternary
derivation and $\delta (1) =0$. \item $\delta$ is a Jordan
derivation.
\end{enumerate} Further, a conjugate-linear mapping $\delta : A \to X$ is a
ternary derivation with $\delta (1) =0$ if, and only if, the
linear mapping $D: A \to X$, $D(a):= \delta (a^*)$, is a Jordan
derivation.
\end{lemma}

\begin{proof} $(a)\Rightarrow (b)$ Since $X$ is a unital real Jordan
$A_{sa}$-module and $\delta(1) =0$, the identity
$$\delta (a\circ b) = \delta \J a1b = \J {\delta(a)}1b +  \J a{\delta(1)}b +  \J a1{\delta(b)} $$
$$= \J {\delta(a)}1b +  \J a1{\delta(b)} = {\delta(a)}\circ b + a\circ
{\delta(b)},$$ gives the desired statement.\smallskip

For every Jordan derivation $\delta : A_{sa} \to X,$ we have
$\delta (1) = \delta (1\circ 1) = 2 (1 \circ \delta (1)) = 2
\delta (1)$, and hence $\delta(1) =0$. The implication
$(b)\Rightarrow (a)$ follows straightforwardly.\smallskip

To prove the last statement, we observe that a conjugate-linear
mapping $\delta : A \to X$ is a ternary derivation with $\delta
(1) =0$ if, and only if, $\delta|_{A_{sa}} : A_{sa} \to X$ is a
ternary derivation with $\delta (1) =0$, which, by $(a)\Leftrightarrow (b)$,
is equivalent to say that
$\delta|_{A_{sa}}$ is a Jordan derivation. It is easy to check
that $\delta|_{A_{sa}} = D|_{A_{sa}}$ is a Jordan derivation if
and only if $D$ is a Jordan derivation from $A$ to $X$.
\end{proof}

Henceforth, given a unital associative *-algebra, $A$, and a Jordan
$A$-module, $X$, we shall write $\mathcal{D}_t^{o}(A,X)$ for the set of all (continuous) ternary derivations
from $A$ to $X$ vanishing at the unit element. We have seen
in Lemma \ref{l ternary and associative derivations} that, when $A$ and $X$ are unital,
we have \begin{equation}\label{equ jordan and ternary vanishing at 1}
\mathcal{D}_J(A,X)\circ * =\mathcal{D}_t^{o}(A,X):=\{\delta\in {\mathcal D}_t(A,X):\delta(1)=0\}.
\end{equation}

Given a Banach *-algebra $A$, we consider the involution $^*$ on
$A^*$ defined by $\varphi^* (a) := \overline{\varphi (a^*)}$
($a\in A$, $\varphi\in A^*$). An element $\delta\in \mathcal{D}_J
(A,A^*)$ is called a \emph{*-derivation} if $\delta(a^*) =
\delta(a)^*$, for every $a\in A$. The symbols $\mathcal{D}_J^*
(A,A^*)$ and $\mathcal{I}nn_J^*(A,A^*)$ (resp., $\mathcal{D}_b^*
(A,A^*)$ and $\mathcal{I}nn_b^*(A,A^*)$) will denote the sets of
all Jordan and Jordan-inner (resp., associative and inner)
*-derivations from $A$ to $A^*$, respectively.

\begin{lemma}\label{l 3.4.2}%lemma 4.2
Let $X$ be an $A$-bimodule over a Banach $^*$-algebra $A$. Then
the following statements hold:
\begin{enumerate}[$(i)$]
\item $\mathcal{I}nn_J (A,X)\subset \mathcal{I}nn_b(A,X)$. In
particular, $\mathcal{I}nn_J^*(A,A^*)\subset
\mathcal{I}nn_b^*(A,A^*)$. \item Let $D$ be an element in
$\mathcal{I}nn_b(A,A^*)$, that is, $D =D_{\varphi}$ for some
$\varphi$ in $A^*$. Then
$D$ is a *-derivation whenever $\varphi^*=-\varphi$. Further, if
the linear span of all commutators of the form $[a,b]$ with $a,b$
in $A$ is norm-dense in $A$, then $D$ is a *-derivation if, and
only if, $\varphi^*=-\varphi$.
\end{enumerate}
\end{lemma}

\begin{proof}
$(i)$ Let us consider a Jordan derivation of the form $\delta_{x_0,b}$,
where $x_0\in X$ and $b\in A$. For each $a$ in $A$, we can easily check that
$$\delta_{x_0,b} (a) = (x_{_0} \circ a)\circ b - (b \circ a)\circ x_{_0}
= \frac14 \left( [b,x_0] a -a [b,x_0]\right)= D_{\frac14 [b,x_0]}
(a),$$ where the Lie bracket $[.,.]$ is defined by $[b,x_0] =
 ( b x_0 - x_0 b)$ for every $b\in A$, $x_0\in X$. Since
every inner Jordan derivation $D$ from $A$ to $X$ must be a finite
sum of the form $D = \sum_{j=1}^{n} \delta_{x_j,b_j},$ with
$x_j\in X$ and $b_j\in A$, it follows that $D= \sum_{j=1}^{n}
D_{\frac14 [b_j,x_j]} = D_{\frac14  \sum_{j=1}^{n} [bj,x_j]}$ is
an inner (associative) binary derivation.\smallskip

$(ii)$ Let $D =D_{\varphi},$ where $\varphi\in A^*$ and
$\varphi^*=-\varphi$. Let us fix two arbitrary elements $a,b$ in
$A$. The identities $$ D_{\varphi} (a^*) (b) = \left( \varphi a^*
- a^* \varphi \right) (b) = \varphi ( a^* b - b a^*) $$ and $$
D_{\varphi} (a)^* (b) = \left( \varphi a - a \varphi \right)^* (b)
= \left(a^* \varphi^* - \varphi^* a^* \right) (b) = \varphi^* (b
a^* - a^* b),$$ give $D_{\varphi} (a^*) = D_{\varphi} (a)^*$,
proving that $D$ is a $^*$-derivation.\smallskip

Conversely, suppose now that the linear span of all commutators of
the form $[a,b]$ with $a,b$ in $A$ is norm-dense in $A$ and $D
=D_{\varphi}$ is a $^*$-derivation. The identity $D_{\varphi}
(a^*) = D_{\varphi} (a)^*$ ($a\in A$), implies that $\varphi
[a^*,b] = -\varphi^* [a^*,b]$, for every $a,b\in A$, therefore
$\varphi= -\varphi^*$ as we wanted.
\end{proof}

\begin{remark}{\rm There exist many examples of Banach algebras $A$ in
which the linear span of all commutators of the form $[a,b]$ with
$a,b$ in $A$ is norm-dense in $A$. This property is never
satisfied by a commutative Banach algebra. However, the
list of examples of C$^*$-algebras satisfying this property
includes all properly infinite C$^*$-algebras, all properly
infinite von Neumann algebras, and the C$^*$-algebra of all compact operators
on an infinite dimensional complex Hilbert space \cite{PeaTop71}
(see also the survey \cite{Weiss04}). T. Fack proved in \cite{Fack} that
if the unit of a (unital) C$^*$-algebra $A$ is properly infinite
(i.e. there exist two orthogonal projections $p, q$ in $A$
Murray-von Neumann equivalent to 1), then any hermitian element is
a sum of at most five self-adjoint commutators. Many other results
have been established to show that all elements in a C$^*$-algebra
which have trace zero with respect to all tracial states can be
written as a sum of finitely many commutators (compare
\cite{Marc02}, \cite{Pop}, \cite{Marc06}, \cite{Marc10}, among
others).}\end{remark}

Let $X$ be a unital Banach $A$-bimodule over a unital Banach
algebra $A$. Regarding $X$ as a real Banach triple $A$-module with
respect to the induced triple product $\J axc = \frac12 ( a x
c+cxa)$,  $\J xac = \frac12 ( xac+cax)$ ($a,c\in A$, $x\in X$),
we can easily see that every ternary derivation $\delta :
A \to X$ annihilates at $1$, that is, $$\mathcal{D}_t
(A,X)=\mathcal{D}_t^{o}(A,X).$$ Indeed, since $$\delta(1) =\delta
(\J 111) = \J {\delta(1)}11 +\J 1{\delta(1)}1+ \J 11{\delta(1)} =
3 \delta(1),$$ we have $\delta(1)=0$. When we consider Banach
$A$-bimodules equipped with ternary products which differ from the
previous one, the identity $\mathcal{D}_t
(A,X)=\mathcal{D}_t^{o}(A,X)$ doesn't hold in general. Our next
lemmas study the case  $X=A^*$, where $A$ is a unital  Banach
$^*$-algebra.

\begin{lemma}
\label{l delta(1) anti-symmetric} Let $A$ be a unital Banach
$^*$-algebra equipped with the ternary product given by $\J abc =
\frac12 \ ( a b^* c + cb^* a) $. Every ternary derivation $\delta
$ in $\mathcal{D}_t (A,A^*)$ satisfies the identity $\delta(1)^*
=- \delta(1)$, that is, $\overline{\delta(1) (a^*)} =- \delta(1)
(a),$ for every $a$ in $A$.
\end{lemma}

\begin{proof} Let $\delta : A\to A^*$ be a ternary derivation.
Since the identity
$$\delta(1) (a) =\delta (\J 111) (a)=\J {\delta(1)}11 (a)+\J 1{\delta(1)}1 (a)+ \J 11{\delta(1)} (a)$$
$$ = 2 \delta(1) \J 11a  + \overline{\delta(1) \J 1a1} = 2 \delta(1) (a) + \delta(1)^* (a),$$
holds for every $a\in A$, we do have $\delta(1)^* =- \delta(1).$
\end{proof}

\begin{lemma}\label{l 3.4.3}%lemma 4.4
 Let $A$ be a unital Banach
$^*$-algebra equipped with the ternary product given by $\J abc =
\frac12 \ ( a b^* c + cb^* a) $. Then
$$\mathcal{D}_t (A,A^*)=\mathcal{D}_t^{o}(A,A^*)+\mathcal{I}nn_t(A,A^*).$$
More precisely, if $\delta \in \mathcal{D}_t(A,A^*)$, then
$\delta=\delta_0+\delta_1$, where $\delta_0\in
\mathcal{D}_t^{o}(A,A^*)$ and $\delta_1$, defined by $\delta_1 (a)
:= \delta(1)\circ a^* = \frac{1}{2}(\delta(1) \ a^* + a^* \
\delta(1))$, is the inner derivation
$-\frac{1}{2}\delta(1,\delta(1))$.
\end{lemma}
\begin{proof} Let $\delta : A\to A^*$ be a ternary derivation. The mapping
$\delta_1 : A \to A^*$ $\delta_1 (a) := \delta(1)\circ a^*$ is a
conjugate-linear mapping with $\delta_1 (1) = \delta(1)$.  We will
show that $\delta_1=-\frac{1}{2}\delta(1,\delta(1))$. Then, the
mapping $\delta_0 = \delta -\delta_1$ is a triple derivation with
$\delta_0 (1) = 0$ and $\delta = \delta_0+ \delta_1$, proving the
lemma.\smallskip

Lemma \ref{l delta(1) anti-symmetric} above implies that
$\delta(1)^* =- \delta(1).$ \smallskip

Now we consider the inner triple derivation $-\frac{1}{2}\delta(1,\delta (1))$. For each $a$ and $b$ in $A$ we have
$$ -\frac{1}{2}\delta(1,\delta(1)) (a) (b) = -\frac{1}{2} \left( \J {1}{\delta(1)}a - \J {\delta(1)}1a\right) (b)$$
$$= -\frac{1}{2} \left( \overline{\delta(1)(\J {1}{b}a)} - \delta(1)(\J 1ab) \right) $$
$$= -\frac{1}{2} \left( \delta(1)^* (\J {1}ab) - \delta(1)(\J 1ab ) \right)$$ $$(\hbox{since } \delta(1)^* =- \delta(1))= -\frac{1}{2} \left(- \delta(1) (\J {1}ab) - \delta(1)(\J 1ab ) \right)$$ $$ = \delta(1)(\J 1ab ) = \delta(1)(a^* \circ b )= \delta_1 (a) (b).$$
Thus, $\delta_1 = -\frac{1}{2}\delta(1,\delta(1))$ as promised.\smallskip
\end{proof}

\begin{lemma}\label{lem:4.5}%3.4.1 now 4.5
Let $A$ be a unital Banach $^*$-algebra equipped with the ternary product given by $\J abc =
\frac12 \ ( a b^* c + cb^* a) $ and the Jordan product $a\circ b=(ab+ba)/2$, let $D:A\rightarrow A^*$
be a linear mapping and let $\delta: A \to A^*$ denote the conjugate linear
mapping defined by $\delta (a) :=D(a^*)$. Then $D$ lies in
${\mathcal D}_J(A,A^*)$ if, and only if, $\delta\J a1b=\J{\delta (
a)}{1}{b}+\J{a}{1}{\delta (b)}$ for all $a,b\in A$. Moreover,
$$\mathcal{D}_t^o(A,A^*)\!=\!\{ \delta: A \rightarrow A^* : \exists
D\in {\mathcal D}_J^*(A,A^*) \hbox{ s.t. }
 \delta (a) = D (a^*),(a\in A)\}.$$
\end{lemma}
\begin{proof}
The first statement follows immediately from the definitions, that
is, $\J{\delta a}{1}{b}=D(a^*)\circ b^*$, $\J{a}{1}{\delta
b}=D(b^*)\circ a^*$, and $\delta\J a1b=D(a^*\circ b^*)$.\smallskip

Suppose next that $\delta\in \mathcal{D}_t^o(A,A^*)$. From  the
first statement, $D$ lies in ${\mathcal D}_J(A,A^*)$.  Actually $D$ is
$^*$-derivation; if $a\in A$ then $\delta(a^*)=\delta\J
1a1=\J{1}{\delta (a)}{1}$, so for all $y\in A$, we have
$$\pair{\delta(a^*)}{y}=\pair{\J{1}{\delta
(a)}{1}}{y}=\overline{\pair{\delta (a)}{\J 1y1}}=\pair{(\delta(
a))^*}{y},$$ and hence $D(a^*)=\delta (a)=(\delta(a^*))^*=(Da)^*$.

Suppose now that $D\in {\mathcal D}_J^*(A,A^*)$.  It follows from
the definitions and the fact that $D\in{\mathcal D}_J(A,A^*)$ that
the following three equations hold:
$$ \delta\J aba=2(D(a^*)\circ b)\circ a^*+2(a^*\circ D(b))\circ a^*+2(a^*\circ b)\circ D(a^*)$$
$$-2(D(a^*)\circ a^*)\circ b-(a^*\circ a^*)\circ D(b),$$
$$ \J{\delta (a)}{b}{a}=D(a^*)\circ(b\circ a^*)+(D(a^*)\circ b)\circ a^*-(D(a^*)\circ a^*)\circ b
$$
and
$$
\J{a}{\delta (b)}{a}=2((D(b^*))^*\circ a^*)\circ a^*-D(b)\circ
(a^*\circ a^*).
$$
From these three equations, we have
$$
\delta\J aba-2\J{\delta (a)}{b}{a}-\J{a}{\delta (b)}{a}=2(a^*\circ
D(b))\circ a^*-2((D(b^*))^*\circ a^*)\circ a^*.
$$
Since $D$ is self-adjoint, the right side of the last equation vanishes, and the result follows.
\end{proof}

\begin{proposition}\label{p 3.4.4}%3.4.4 now 4.6
Let $A$ be a unital Banach $^*$-algebra equipped with the ternary product given by $\J abc =
\frac12 \ ( a b^* c + cb^* a) $ and the Jordan product $a\circ b=(ab+ba)/2$. Then
$$\mathcal{D}_t(A,A^*)\subset \mathcal{D}_J^*(A,A^*)\circ *+\mathcal{I}nn_t(A,A^*).$$
If $A$ is Jordan weakly amenable, then
$$ \mathcal{D}_t(A,A^*)=\mathcal{I}nn_b^*(A,A^*)\circ *+\mathcal{I}nn_t(A,A^*).$$
\end{proposition}

\begin{proof}
Let $\delta : A \to A^*$ be a ternary derivation. By Lemma \ref{l 3.4.3}, $\delta=\delta_0+\delta_1$, where
$\delta_0\in \mathcal{D}_t^{o}(A,A^*)$, $\delta_1 (a)= -\frac{1}{2}\delta(1,\delta(1)) (a) = \delta(1)\circ a^*$.
Lemmas~\ref{l ternary and associative derivations} and \ref{lem:4.5}
assure that $D = \delta_0 \circ *$, is a Jordan $^*$-derivation. This proves the first statement.

The assumed Jordan weak amenability of $A$, together with Lemma~\ref{l 3.4.2}
implies that $D = \delta_0 \circ *$ lies in $\mathcal{I}nn_b^*(A,A^*)$, which gives
$\delta = D\circ * +\delta_1 \in \mathcal{I}nn_b^*(A,A^*)\circ *+\mathcal{I}nn_t(A,A^*)$.
Since a simple calculation shows that $\mathcal{I}nn_b^*(A,A^*)\subset \mathcal{D}_t(A,A^*)$,
the reverse inclusion holds, proving the  second statement.\end{proof}

When a Banach $^*$-algebra $A$ is  commutative, we have $\mathcal{I}nn_b(A,A^*) =\{0\}.$ In the setting of unital and commutative Banach $^*$-algebras,
the above Proposition \ref{p 3.4.4} implies the following:

\begin{corollary}\label{c 3.4.5}
Let $A$ be a unital and commutative Banach $^*$-algebra.
Then $A$ is ternary weakly amenable whenever it is Jordan weakly amenable.$\hfill\Box$
\end{corollary}

Every C$^*$-algebra $A$ is binary weakly
amenable (cf. \cite{Haa83}), and by \cite[Theorem 19 or Corollary 21]{PeRu},
every Jordan derivation $D: A \to A^*$ is continuous,
and hence an associative derivation by Johnson's Theorem \cite{John96}.
This gives us the next corollary.

\begin{corollary}\label{c 3.4.6}%4.8
Every unital and commutative (real or complex) C$^*$-algebra is ternary weakly amenable.$\hfill\Box$
\end{corollary}

The following corollary of Proposition~\ref{p 3.4.4} will be used in the next section. The proof consists in observing that
Lemmas~\ref{lem:4.1}, \ref{l 3.4.3} and \ref{lem:4.5} are valid in this context and using
\cite[Theorem 3.2]{BunPasch}.

\begin{corollary}\label{cor:4.9}%bunce paschke 4.9
Let $M$ be a semifinite von Neumann algebra and consider the submodule $M_*\subset M^*$.  Then
$$
{\mathcal D}_t(M,M_*)={\mathcal I}nn_b^*(M,M_*)\circ *+{\mathcal I}nn_t(M,M_*).
$$
\end{corollary}

Our next proposition shows that Corollary~\ref{c 3.4.6}  remains valid in the setting of
non-necessarily-unital abelian C$^*$-algebras.

\begin{proposition}\label{p 3.4.7}
Every commutative (real or complex) C$^*$-algebra is \linebreak
ternary weakly amenable.
\end{proposition}

\begin{proof}
Let $A$ be a commutative C$^*$-algebra and let $\delta: A \to A^*$
be a ternary derivation. By Proposition \ref{p bidual extensions},
$\delta^{**} : A^{**} \to A^{***}$ is a weak$^*$-continuous
ternary derivation with $\delta^{**} (A^{**}) \subseteq A^{*}$.
Since $A^{**}$ is a unital and commutative (real or complex)
C$^*$-algebra,
$\mathcal{D}_t(A^{**},A^{***})=\mathcal{I}nn_t(A^{**},A^{***})$
and $\delta^{**}$ may be written in the form $ \delta^{**}=
-\frac{1}{2}\delta(1,\delta^{**}(1))$ (compare Corollary \ref{c
3.4.6} and Lemma \ref{l 3.4.3}).\smallskip

Since every C$^*$-algebra admits a bounded approximate unit (cf.
\cite[Theorem 1.4.2]{Ped}), by Cohen's Factorisation Theorem (cf.
\cite[Theorem VIII.32.22]{HewRoss}), there exist $b\in A$ and
$\varphi\in A^*$ such that $\frac12 \delta^{**} (1) = \varphi b$.
Finally, for each $a$ in $A$ we have
$$ \delta (a) = -\frac{1}{2}\delta(1,\delta^{**}(1)) (a)
=\delta(1,\varphi b) (a) = \J {1}{\varphi b}{a} - \J {\varphi
b}{1}{a} $$
$$= \hbox{ (by the commutativity of $A$) }
= \J {b}{\varphi}{a} - \J {\varphi}{b}{a}
= \delta (b,\varphi) ({a}),$$ which gives $\delta =  \delta (b,\varphi)$.
\end{proof}

\section{C$^*$-algebras are not ternary weakly amenable}

In this section we present some examples of C$^*$-algebras
which are not ternary weakly amenable.

\begin{lemma}\label{lem:4.11}%3.5.1 now 4.11
The C$^*$-algebra $A=K(H)$ of all compact operators on an
infinite dimensional Hilbert space $H$  is not Jordan weakly amenable.
\end{lemma}

\begin{proof}
%By the theorems of Johnson and Haagerup referred to several times already, we have
%$$ {\mathcal D}_J(A,A^*)={\mathcal D}_b(A,A^*)={\mathcal I}nn_b(A,A^*).$$
We shall identify $A^*$ with the trace-class operators on $H$.\smallskip

Supposing that $A$ were Jordan weakly amenable, let $\psi\in A^*$ be arbitrary.
Then $D_\psi$ would be an inner Jordan derivation, so there would exist $\varphi_j\in A^*$
and $b_j\in A$ such that $D_\psi(x)=\sum_{j=1}^n[\varphi_j\circ(b_j\circ x)-b_j\circ(\varphi_j\circ x)]$
 for all $x\in A$.\smallskip

For $x,y\in A$, a direct calculation yields
$$ \psi(xy-yx)=-\frac{1}{4}\left(\sum_{j=1}^n b_j \varphi_j-\varphi_j b_j\right)(xy-yx). $$\smallskip

It is known (\cite[Theorem 1]{PeaTop71}) (see also the excellent survey \cite{Weiss04})
that every compact operator on a separable infinite dimensional Hilbert space is
a finite sum of commutators of compact operators. Let $z$ be any element in $A= K(H)$.
By standard spectral theory, we can find a separable infinite dimensional Hilbert
subspace $H_0\subseteq H$ such that $z\in K(H_0)$,
that is, $z = p z = z p$, where $p$ is the orthogonal projection of $H$ onto $H_0$.
By the just quoted theorem of Pearcy and Topping, $z$ can be written as a
finite sum of commutators $[x,y]= xy -yx$ of elements $x,y$ in $K(H_0) = p K(H) p\subseteq K(H).$
Thus, it follows that the trace-class operator
$\psi = -\frac{1}{4}\left(\sum_{j=1}^n b_j \varphi_j-\varphi_j b_j\right)$ is a finite sum of
commutators of compact and trace-class operators, and hence has trace zero.
This is a contradiction, since $\psi$ was arbitrary.
\end{proof}

\begin{proposition}\label{p compact operators is not ternary weakly amen}%3.5.2 now 4.12
The C$^*$-algebra $A=K(H)$ of all compact operators on
an infinite dimensional Hilbert space $H$  is not ternary weakly amenable.
\end{proposition}

\begin{proof}
Let $\psi$ be an arbitrary element in $A^*$.
The binary inner derivation $D_{\psi}: x\mapsto \psi x-x \psi$
may be viewed as a map from either $A$ or $A^{**}$ into $A^*$.
Considered as a map on $A^{**}$, it belongs to ${\mathcal I}nn_b(A^{**},A^*)$
so by Corollary~\ref{cor:4.9},  $D_\psi\circ *: a\mapsto D_\psi (a^*),$
belongs to ${\mathcal D}_t(A^{**},A^*)$.\smallskip

Assuming that $A$ is ternary weakly amenable, the restriction of
$D_\psi\circ *$ to $A$ belongs to ${\mathcal I}nn_t(A,A^*)$.
Thus, there exist $\varphi_j\in A^*$ and $b_j\in A$ such that
$D_\psi\circ *=\sum_{j=1}^n (L(\varphi_j,b_j)-L(b_j,\varphi_j))$ on $A$.\smallskip

For $x,a\in A$, direct calculations yield
$$ \psi(a^*x-xa^*)=\frac{1}{2}\sum_{j=1}^n(\varphi_jb_j-b_j^*\varphi_j^*)(a^*x)
+\frac{1}{2}\sum_{j=1}^n(b_j\varphi_j-\varphi_j^*b_j^*)(xa^*).$$

We can and do set $x=1$  to get
\begin{equation}\label{eq:6810}
0=\frac{1}{2}\sum_{j=1}^n(\varphi_jb_j-b_j^*\varphi_j^*)(a^*)
+\frac{1}{2}\sum_{j=1}^n(b_j\varphi_j-\varphi_j^*b_j^*)(a^*),
\end{equation}
and therefore
\begin{equation}\label{eq:681}\psi(a^*x-xa^*)=\frac{1}{2}
\sum_{j=1}^n(\varphi_jb_j-b_j^*\varphi_j^*)(a^*x-xa^*),\end{equation} for every $a,x\in A$.\smallskip

Using \cite[Theorem 1]{PeaTop71} as in the proof of Lemma~\ref{lem:4.11}, and taking note of (\ref{eq:681}) and (\ref{eq:6810}),
we have
$$\psi=\frac{1}{2}\sum_{j=1}^n(\varphi_jb_j-b_j^*\varphi_j^*)=\frac{1}{2}\sum_{j=1}^n(\varphi_j^*b_j^*-b_j\varphi_j).$$
Hence
\begin{eqnarray*}
2\psi&=&\sum_{j=1}^n(\varphi_jb_j-b_j\varphi_j+b_j\varphi_j-\varphi_j^*b_j^*+\varphi_j^*b_j^*-b_j^*\varphi_j^*)\\
&=&\sum_{j=1}^n[\varphi_j,b_j]-2\psi+\sum_{j=1}^n[\varphi_j^*,b_j^*].
\end{eqnarray*}
Finally, the argument given at the end of the proof of Lemma~\ref{lem:4.11} shows that $\psi$ has trace 0,
which is a contradiction, since $\psi$ was arbitrary.
\end{proof}

Next we study the ternary weak amenability of the C$^*$-algebra $L(H)$ of all bounded
linear operators on a complex Hilbert space $H$. We shall recall first some standard
theory of von Neumann algebras.\smallskip

Given a von Neumann algebra $M$, with predual $M_*$ and dual $M^*$, there exists a (unique) central
projection $z_0$ in $M^{**}$ satisfying  $M_* = M^{*} z_0$. Moreover, denoting
$M_*^{\perp} = M^* (1-z_0)$ we have $M^{*} = M_* \oplus^{\ell_1} M_*^{\perp}$
(cf. \cite[Theorem III.2.14]{Tak}). $M_*$ (resp., $M_*^{\perp}$) is called the \emph{normal}
(resp., the \emph{singular}) part of $M^*$. Every functional $\phi$ in $M^*$ is uniquely decomposed into
the sum $$\phi = \phi_n  + \phi_s, \ \ \phi_n \in M_*, \ \phi_s\in M_*^{\perp}.$$
The functionals $\phi_n$  and $\phi_s$ are called respectively, the normal part and the
singular part of $\phi$. Since $z_0$ is a central projection in $M^{**}$, we can easily see that $$\left(\phi a\right)_n = \phi_n a, \ \left(\phi a\right)_s = \phi_s a, \ \left(a\phi \right)_n =a \phi_n , \ \left(a \phi \right)_s =a \phi_s, $$ $$\J {\phi}{a}{b}_n = \J {\phi_n}{a}{b}, \ \J {\phi}{a}{b}_s = \J {\phi_s}{a}{b},$$ $$\J {a}{\phi}{b}_n = \J {a}{\phi_n}{b},\ \hbox{and } \J {a}{\phi}{b}_s = \J {a}{\phi_s}{b},$$ for every $a,b\in M$ and $\phi\in M^{*}$.

\begin{lemma}\label{lem:4.11 L(H)}%3.5.1 now 4.11
The C$^*$-algebra $M=L(H)$ of all bounded linear operators on an
infinite dimensional Hilbert space $H$ is not Jordan weakly amenable.
\end{lemma}

\begin{proof} Let $B= K(H)$ denote the ideal of all compact operators on $H$. We notice
that $B^{**} = M =L(H)$ and hence $M_* = B^*$ coincides with the trace-class operators on $H$.
Let $\psi$ be an element in $B^*$ whose
trace is not zero. The argument given in the proof of Lemma \ref{lem:4.11} guarantees that
the derivation $D_\psi : B \to B^*$, $a\mapsto \psi a - a \psi$ doesn't belong to
${\mathcal I}nn_J(B,B^*)$.\smallskip

By Proposition \ref{p bidual extensions} and its proof,  $D_\psi^{**} : B^{**}=M \to B^*=M_*\subseteq M^*$ is a Jordan
derivation whose image is contained in $M_*$. It can be easily checked that $D_\psi^{**} (x) = \psi x - x \psi,$
for every $x\in B^{**} =M$. We claim that $D_\psi^{**}$ is not an inner Jordan derivation. Otherwise, there exist $\varphi_j\in M^*$
and $b_j\in M$ such that $D_\psi^{**} (x)=\sum_{j=1}^n[\varphi_j\circ(b_j\circ x)-b_j\circ(\varphi_j\circ x)],$ for all $x\in M$.
For each $j$, let us write $\varphi_j = \phi_j + \psi_j$, where $\phi_j\in M_*$ and $\psi_j\in M_*^{\perp}$ are the normal
and singular part of $\varphi$, respectively. Since  $D_\psi^{**}$ remains $M_*$-valued and, for each $x$ in $M$, $\sum_{j=1}^n[\phi_j\circ(b_j\circ x)-b_j\circ(\phi_j\circ x)]\in M_*$ and $\sum_{j=1}^n[\psi_j\circ(b_j\circ x)-b_j\circ(\psi_j\circ x)]\in M_*^{\perp}$, it follows that $D_\psi^{**} (x)=\sum_{j=1}^n[\phi_j\circ(b_j\circ x)-b_j\circ(\phi_j\circ x)],$ for all $x\in M$, where, in this case, $\phi_j\in M_*$ and $b_j\in M$.\smallskip

Now, we can mimic the argument in the proof of Lemma \ref{lem:4.11} to show that
$\psi = -\frac{1}{4}\left(\sum_{j=1}^n b_j \phi_j-\phi_j b_j\right)$ is a finite sum of
commutators of bounded and trace-class operators, and hence has trace zero, which is impossible.
\end{proof}

Similar ideas to those applied in the previous lemma give us the following result.

\begin{proposition}\label{p bd operators is not ternary weakly amen}%3.5.2 now 4.12
The C$^*$-algebra $M=L(H)$ of all bounded linear operators on
an infinite dimensional Hilbert space $H$ is not ternary weakly amenable.
\end{proposition}

\begin{proof} Let $B= K(H)$ denote the ideal of all compact operators on $H$
(notice that $B^{**} =L(H)=M$).
Let $\psi$ be an element in $B^*$ whose trace is not zero. From Proposition
\ref{p compact operators is not ternary weakly amen} and its proof we know that
the mapping $D_\psi \circ *: B \to B^*$, $a\mapsto \psi a^* - a^* \psi$ is a
ternary derivation (see Corollary~\ref{cor:4.9}) which doesn't belong to
${\mathcal I}nn_t (B,B^*)$.\smallskip

Applying Proposition \ref{p bidual extensions} and its proof, the bitranspose
$D_\psi^{**} : B^{**}=M \to B^*=M_*\subseteq M^*$ is an associative
derivation whose image is contained in $M_*$. Moreover, $D_\psi^{**} (x) = \psi x - x \psi,$
for every $x\in B^{**} =M$. We shall prove that $D_\psi^{**}\circ *$ is ternary
derivation from $M$ to $M^*$ (cf. Corollary~\ref{cor:4.9}) which is not inner.
Suppose, on the contrary, that there exist $\varphi_j\in M^*$
and $b_j\in M$ such that $D_\psi^{**} \circ *=\sum_{j=1}^n (L(\varphi_j,b_j)-L(b_j,\varphi_j))$ on $M$.\smallskip

For each $j$, we write $\varphi_j = \phi_j + \psi_j$, where $\phi_j\in M_*$ and $\psi_j\in M_*^{\perp}$ are the normal
and singular part of $\varphi$, respectively. Since  $D_\psi^{**} (M)\subseteq M_*$, and for each $x\in M$,
$\sum_{j=1}^n \J{\phi_j}{b_j}{x}-\J {b_j}{\phi_j}{x} \in M_*$ and $\sum_{j=1}^n \J{\psi_j}{b_j}{x}-\J {b_j}{\psi_j}{x} \in M_*^{\perp}$
we have $D_\psi^{**} \circ *=\sum_{j=1}^n (L(\phi_j,b_j)-L(b_j,\phi_j))$ on $M$, where $\phi_j\in M_*$ and $b_j\in M$.\smallskip

Following the lines in the last part of the proof of Proposition \ref{p compact operators is not ternary weakly amen}
we derive at \begin{eqnarray*}
4 \psi&=&\sum_{j=1}^n[\phi_j,b_j]+\sum_{j=1}^n[\phi_j^*,b_j^*],
\end{eqnarray*} which is impossible because $\psi$ has non-zero trace.
\end{proof}

The techniques in this subsection can be used to show that the Cartan factor $M_n(\CC)$ of all  operators on a finite dimensional Hilbert space is ternary weakly amenable.
This is of course a special case of Proposition~\ref{p finite dimensional} below, but we are able to give a direct proof here.  Note also that  $M_n(\CC)$ is Jordan weakly amenable by \cite{Jacobson51}, or \cite{Jac}.

Although Lemma~\ref{lem:4.11} states that ${\mathcal
I}nn_J(A,A^*)\not={\mathcal D}_J(A,A^*)$ when $A=K(H)$ with $H$ infinite dimensional, nevertheless for $A=M_n(\CC)$,  we shall show directly in the next lemma that the equality does hold for the subsets of
$^*$-derivations.   This is included in the next lemma (as well as following from \cite{Jacobson51} or \cite{Jac}).

\begin{lemma}\label{lem:4.13}%3.5.3 now 4.13
Let $A$ denote the JB$^*$-triple $M_n(\CC)$. Then $${\mathcal
I}nn_b^*(A,A^*)= {\mathcal I}nn_J^*(A,A^*)={\mathcal
D}_J^*(A,A^*)$$
\end{lemma}
\begin{proof}
Let $D\in {\mathcal I}nn_b^*(A,A^*)$ so that $D(x)=\psi x-x\psi$
for some $\psi\in A^*$. Recall (\cite[Theorem 1]{PeaTop71}) that
every compact operator is a finite sum of commutators of compact
operators.  Therefore, by Lemma~\ref{l 3.4.2}(ii), $\psi^*=-\psi$.
Also, since every matrix of trace 0 is a commutator
(\cite{Shoda37}, and \cite{AlbMuc57}), we have
$\psi=[\varphi,b]+\frac{\hbox{Tr}\, (\psi)}{n} I$.  Expanding
$\varphi=\varphi_1+i\varphi_2$ and $b=b_1+ib_2$ into hermitian and
skew symmetric parts and using $\psi^*=-\psi$ leads to
$$\psi=[\varphi_1,b_1]-[\varphi_2,b_2]+\frac{\hbox{Tr}\, (\psi)}{n}
I.$$

For $x,y \in A$, direct calculation yields
$$D(x)=\varphi_1\circ(b_1\circ x)-b_1\circ (\varphi_1\circ
x)-\varphi_2\circ(b_1\circ x)+b_2\circ(\varphi_2\circ x),$$ so
that $D\in{\mathcal I}nn_J^*(A,A^*)$.\smallskip

From the theorems of Haagerup (alternatively  \cite[Th. 2.2]{Hochs}) and Johnson,
and what was just proved, we have $${\mathcal D}_J^*(A,A^*)={\mathcal
D}_b^*(A,A^*)={\mathcal I}nn_b^*(A,A^*)\subseteq {\mathcal
I}nn_J^*(A,A^*)\subseteq {\mathcal D}_J^*(A,A^*).$$\end{proof}

\begin{proposition}%3.5.5 now 4.14
The JB$^*$-triple $A=M_n(\CC)$ is ternary weakly amenable and Jordan weakly amenable.
\end{proposition}
\begin{proof}
We have noted above that $M_n(\CC)$ is Jordan weakly amenable.

By Proposition~\ref{p 3.4.4}
$$ \mathcal{D}_t(A,A^*)=\mathcal{I}nn_b^*(A,A^*)\circ *+\mathcal{I}nn_t(A,A^*),$$
so it suffices to prove that $\mathcal{I}nn_b^*(A,A^*)\circ *\subset
\mathcal{I}nn_t(A,A^*).$\smallskip

As in the proof of Lemma~\ref{lem:4.13}, if $D\in {\mathcal
I}nn_b^*(A,A^*)$ so that $Dx=\psi x-x\psi$ for some $\psi\in A^*$,
then $\psi=[\varphi_1,b_1]-[\varphi_2,b_2]+\frac{\hbox{Tr}\,
(\psi)}{n}I$, where $b_1,b_2$ are self adjoint elements of $A$ and
$\varphi_1$ and $\varphi_2$ are self adjoint elements of $A^*$. It
is easy to see that, for each $x\in A$, we have
$$D(x^*)=\J{\varphi_1}{2b_1}{x}-\J{2b_1}{\varphi_1}{x}-\J{\varphi_2}{2b_2}{x}
+\J{2b_2}{\varphi_2}{x},$$ so that $D\circ *\in
\mathcal{I}nn_t(A,A^*).$
\end{proof}

\section{The case of Cartan factors}

Contrary to what happens for (binary) weak amenability
in the setting of C$^*$-algebras,
not every JB$^*$-triple is ternary weakly amenable.
In this section, we shall study weak amenability
for some examples of JB$^*$-triples.\smallskip

As was mentioned in the introduction, every finite dimensional JB$^*$-triple
has the inner derivation property (cf. \cite[Chpt.11]{Meyberg72} or  \cite[Chpt. 8]{Loos77}),
and indeed is ``super amenable'', meaning that every derivation
into a Jordan Banach triple module is inner (see \cite[III.Korollar 1.6]{KR78}).
In particular, we have the following proposition.

\begin{proposition}\label{p finite dimensional} %prop 3.9
Every finite dimensional JB$^*$-triple
is ternary weakly \linebreak amenable.
\end{proposition}

\subsection{Hilbert spaces and finite rank type I Cartan factors}

Let $X$ be a real Hilbert space considered as a real Cartan factor of type I, with respect to the product \begin{equation}\label{eq:0517121}
\J xyz :=\frac12( (x|y) z + (z|y) x) \ \ (x,y,z\in X),\end{equation} where $(.|.)$ denotes the inner product of $X$. Henceforth, $J= J_{X}: X \to X^*$ will denote the Riesz mapping. We begin with a useful observation.

\begin{proposition}\label{p derivations from real Hibert space to its dual} Let $\delta : X \to X^*$ be linear mapping. Then denoting
$T = J^{-1} \delta: X \to X$, the following are equivalent:\begin{enumerate}[$(a)$]
\item $\delta$ is a ternary derivation;
\item $T$ is a bounded linear operator with $T^{*} = -T$.
\end{enumerate}
\end{proposition}

\begin{proof}
$(a)\Rightarrow (b)$ By \cite[Corollary 15]{PeRu}, we may assume that $\delta$ (and hence $T$) is continuous.
Let us suppose that $\delta$ is a ternary derivation. For each $x,y,$ and $z$ in $X$ we have
\begin{equation}
\label{eq 1 real Hilbert} \delta \J xyz = \J {\delta (x)}{y}{z} + \J {x}{\delta (y)}{z} + \J xy{\delta (z)}.
\end{equation}
Applying the definition (\ref{eq:0517121}) to (\ref{eq 1 real Hilbert}) results in
$$0= \frac12 \delta(x) (y) J(z)+ \frac12 \delta(y) (z) J (x) + \frac12 \delta(y) (x) J (z)
+ \frac12 \delta(z) (y) J(x),$$ for every $x,y,z\in X.$ Taking $x=z$ we see that $$0= (\delta(x) (y) + \delta(y) (x)) J (x),$$
for every $x,y\in X,$ which gives $\delta(x) (y) + \delta(y) (x)=0,$ for any $x,y\in X$, or equivalently,
$(y|T(x)) = -(x|T(y))=-(T(y)|x) $,  for any $x,y\in X$, which proves $(b)$.\smallskip

$(b)\Rightarrow (a)$
By a direct calculation using (\ref{eq:0517121}) and the definition of $T$, we have
 $\delta \J xyz = \J {\delta (x)}{y}{z} + \J {x}{\delta (y)}{z} + \J xy{\delta (z)}.$
\end{proof}

In the terminology employed above, let $x,y$ be two elements in $X$.
It is not hard to see that the inner derivation $\delta(J (x),y)=L(J(x),y)-L(y,J(x)) : X\to X^*$
is the mapping given by
$\delta(J (x),y) (a) = \frac12 (a|y) J(x) -\frac12 (a|x) J(y)$.
Therefore every inner derivation from $X$ to $X^*$ is a finite rank operator. This argument shows the following:

\begin{corollary}\label{c inner derivations real Hilbert} Let $\delta : X \to X^*$ be a linear mapping. Then denoting
$T = J^{-1} \delta: X \to X$, the following are equivalent:\begin{enumerate}[$(a)$]
\item $\delta$ is a ternary inner derivation;
\item $T$ is a finite rank operator with $T^{*} = -T$.
\end{enumerate}
\end{corollary}

We can prove now that every infinite dimensional real Hilbert space is not ternary weakly amenable.

\begin{proposition}\label{prop:3.5 real}%prop 3.1.5 now 3.5
A real or complex  Hilbert space $X$ regarded as a type I Cartan factor
is ternary weakly amenable if, and only if,
it is finite dimensional.
\end{proposition}

\begin{proof} The  if implication follows from Propositions~\ref{p finite dimensional} and \ref{p complexification}. To see the other implication, suppose $X$ is infinite dimensional.
Then we can find a bounded linear operator $T: X\to X$ having infinite dimensional range and satisfying $T^* = -T$.
Proposition \ref{p derivations from real Hibert space to its dual} and Corollary \ref{c inner derivations real Hilbert}
imply that $\delta = J_{X} T$ is a ternary derivation which is not inner.
\end{proof}

The above results also give new ideas to deal with the ternary weak amenability in other Cartan factors of type I.\smallskip

Suppose $H_1$ and $H_2$ are Hilbert spaces. The symbol $K(H_1,H_2)$ will denote the set
of all compact operators from $H_1$ to $H_2$.
It is known that every $a$ in $K(H_1,H_2)$
can be written  (uniquely) as a (possibly finite) sum of the form $$a= \sum_{n=1}^{\infty} \sigma_n(a) k_n\otimes h_n,$$
where $(\sigma_n(a))\subset \mathbb{R}_0^{+}$ is the sequence of \emph{singular values} of $a$, $(h_n)$ and $(k_n)$
are orthonormal systems in $H_1$ and $H_2$, respectively, and given $\xi\in H_2,$ $\eta,h\in H_1$,
we denote $\xi \otimes \eta (h)= (h|\eta) \xi$ (cf. \cite[\S 1.2]{Sim}).
We denote by $S^1(H_1,H_2)$ the set of all compact operators
$\phi$ from $H_1$ to $H_2$, whose sequence of singular values
$(\sigma_i(\phi))_{i \in \NN}\in \mathbb{R}_0^{+}$ lies in $\ell_1$. For each $\xi\in H_2,$ and $\eta\in H_1$ we can define an element $\omega_{\xi,\eta}\in K(H_1,H_2)^*$, given by $\omega_{\xi,\eta} (x) = (x(\eta)|\xi)$ ($\forall x\in K(H_1,H_2)$). When we equip $S^1(H_1,H_2)$ with the norm
$\|\phi\|_1 =  \sum_i \sigma_i(\phi) $, $(S^1(H_1,H_2),\|.\|_1)$ is a Banach space and $S^1(H_1,H_2)$
can be identified with $K(H_1,H_2)^*$, via the assignment $$\xi \otimes \eta \mapsto \omega_{\xi,\eta}$$ (cf. \cite{Sim}).  We omit the straightforward proof of the following lemma.
%Clearly, $S^1(H,H) = S^1(H)$, is the space of all trace-class operators on $H$.\smallskip

%The next lemma will be required later, the details of the proof are left for the reader.

\begin{lemma}\label{l tech polar} Let $X$ and $Y$ be two real Hilbert spaces. Suppose that $Y_1$ and $Y_2$
are two closed subspaces of $Y$ such that $Y= Y_1 \oplus^{\perp} Y_2$. Then the polar, $K(X,Y_1)^{\circ}$,
of $K(X,Y_1)$ in $K(X,Y)^*= S^{1} (X,Y)$ coincides with $S^1(X,Y_2)$.$\hfill\Box$
\end{lemma}

The following easily verified formulas will facilitate the proof of the next theorem. For real Hilbert spaces $X$ and $Y$, and vectors $\xi,a,c\in Y$ and $\eta,b,d\in X$,

$$2\J{\omega_{\xi,\eta}}{a\otimes b}{c\otimes d}=\ip{b}{d}\ip{\xi}{a}\omega_{c,\eta}+\ip{a}{c}\ip{\eta}{b}\omega_{\xi,d}$$
and
$$2\J{a\otimes b}{\omega_{\xi,\eta}}{c\otimes d}=\ip{\eta}{d}\ip{a}{\xi}\omega_{c,b}+\ip{\eta}{b}\ip{\xi}{c}\omega_{a,d}.$$

\begin{theorem}
\label{th TWA in L(H,K)} Let $X$ and $Y$ be real Hilbert spaces with $\dim (Y)<\infty = \dim (X)$.
Then the real Cartan factor $L(X,Y)$ is not ternary weakly amenable.
\end{theorem}

\begin{proof} Since $\dim (Y)<\infty$, $L(X,Y)= K(X,Y) = S^{1} (X,Y)$ as linear spaces and $L(X,Y) = K(X,Y)$ as Banach spaces.
We can also pick $T\in L(X)$ with infinite dimensional range and $T^* = -T$.
Since the elements $k\otimes h$ (resp., $\omega_{k,h}$) with $h\in X$, $k\in Y$ generate the whole $L(X,Y)$ (resp., $S^{1} (X,Y)$),
the assignment $k\otimes h\mapsto \delta(k\otimes h):= \omega_{k,T(h)}$ defines a linear operator $\delta: L(X,Y) \to L^{1} (X,Y)= K(X,Y)^*$.
We claim that $\delta$ is a ternary derivation. Indeed, it is enough to prove that $$\delta \J {k_1\otimes h_1}{k_2\otimes h_2}{k_3\otimes h_3} =
 \J {\delta(k_1\otimes h_1)}{k_2\otimes h_2}{k_3\otimes h_3} $$ $$+  \J {k_1\otimes h_1}{\delta(k_2\otimes h_2)}{k_3\otimes h_3} +
  \J {k_1\otimes h_1}{k_2\otimes h_2}{\delta(k_3\otimes h_3)},$$ for every $k_1,k_2,k_3\in Y$, $h_1,h_2,h_3\in X$, which follows by direct calculation.\smallskip

We shall finally prove that $\delta$ is not inner. Suppose on the contrary that $\delta = \sum_{j=1}^{p} \delta(\phi_j,a_j)$ for suitable $\phi_1,\ldots,\phi_p\in S^{1} (X,Y)$, $a_1,\ldots,a_p\in K(X,Y)$. Let us fix a norm-one element $k_0\in Y$ and an arbitrary $h\in X$, so that we have $\delta (k_0\otimes h) = \omega_{k_0,T(h)}$. On the other hand, each $\phi_j$ can be written in the form $$\phi_j =  \sum_{n=1}^{m_j} \alpha_{n}^{j} \omega_{k_n^{j},h_n^{j}},$$
where $m_j\leq \dim (Y)$, $\alpha_{n}^{j}>0$ and $(k_n^{j})_{n},$ and $(h_n^{j})_{n}$ are orthonormal systems in $Y$ and $X$, respectively. Now, we can check that

 $$
\omega_{k_0,T(h)}= \delta (k_0\otimes h) = \sum_{j=1}^{p} \delta(\phi_j,a_j) (k_0\otimes h) = \sum_{j=1}^{p} \delta\left(\sum_{n=1}^{m_j} \alpha_{n}^{j} \omega_{k_n^{j},h_n^{j}},a_j\right) (k_0\otimes h)
$$

$$=\sum_{j=1}^{p}  \sum_{n=1}^{m_j} \alpha_{n}^{j} \left(
\frac12 (a_j(h_n^j)|k_0)\ \omega_{k_n^j,h}
+ \frac12 (a_j(h)|k_n^j) \ \omega_{k_0,h_n^j}\right)
$$

$$+\sum_{j=1}^{p}  \sum_{n=1}^{m_j} \alpha_{n}^{j} \left(
-\frac12 (h_n^j|h) \ \omega_{k_0,a_j^*(k_n^j)}
-\frac12 (k_0|k_n^j)\ \omega_{a_j(h_n^j),h} \right).
$$

Write $Y = \mathbb{R} k_0 \oplus^{\perp} Y_2$, where $Y_2 =\{ k_0\}^{\perp}$. Since for every $\xi\in X,$ $\omega_{k_0,\xi}$ lies in $K(X,Y_2)^{\circ} = S^{1} (X, \mathbb{R} k_0)$ (compare Lemma \ref{l tech polar}), it follows from the above identities that the functional

$$\psi:=\sum_{j=1}^{p}  \sum_{n=1}^{m_j} \left(\frac12 (a_j(h_n^j)|k_0)\ \omega_{k_n^j,h} -\frac12 \alpha_{n}^{j} (k_0|k_n^j)\ \omega_{a_j(h_n^j),h}\right) $$
belongs to $K(X,Y_2)^{\circ} = S^{1} (X, \mathbb{R} k_0)$.

Therefore, there exist an scalar $\lambda$ such that $\psi= \omega_{\lambda k_0,h}= \lambda \omega_{ k_0,h}.$ Thus, $$
\omega_{k_0,T(h)}=\lambda \omega_{ k_0,h} + \sum_{j=1}^{p}  \sum_{n=1}^{m_j} \alpha_{n}^{j} \left( \frac12 (a_j(h)|k_n^j) \ \omega_{k_0,h_n^j}+\frac12 (h_n^j|h) \ \omega_{k_0,a_j^*(k_n^j)}\right).$$ In particular, $$T(h)=\lambda h + \sum_{j=1}^{p}  \sum_{n=1}^{m_j} \alpha_{n}^{j} \left( \frac12 (a_j(h)|k_n^j) \ {h_n^j}+\frac12 (h_n^j|h) \ {a_j^*(k_n^j)}\right).$$ Since $h$ was arbitrary, $T$ is a multiple of the identity plus a finite rank operator, that is, $T= \lambda Id_{X} + F,$ where $F:X \to X$ is a finite rank operator. Finally, applying that $T^{*} = -T$, we get $\lambda =0$, and hence $T=F$ is a finite rank operator, which is impossible.
\end{proof}

Let $H$ and $K$ be two complex Hilbert spaces. Every rectangular
complex Cartan factor of type I of the form $L(H,K)$ with $\dim (H)= \infty >\dim(K)$
admits a real form which coincides with $L(X,Y)$, where $X$ and $Y$ are real Hilbert spaces with
$\dim (X)= \infty >\dim(Y)$ (compare \cite{Ka97}). The following corollary follows straightforwardly
from Proposition \ref{p complexification} and Theorem \ref{th TWA in L(H,K)}.

\begin{corollary}
\label{c rectangular Cartan factors of type 1} Let $H$ and $K$ be two complex Hilbert
spaces with $\dim (H)= \infty >\dim(K)$.
Then the rectangular complex Cartan factor of type I, $L(H,K)$, and all its real forms
are not ternary weakly amenable.
\end{corollary}

\subsection{Spin factors}

A (complex) JB$^*$-triple $A$ which can be equipped with an inner product $\ip{\cdot}{\cdot}$
and a conjugation $^\sharp$, satisfying
\begin{enumerate}[$(a)$]
\item the norm on $A$ is given by $ \|x\|^2 =  \ip{x}{x} + \sqrt{\ip{x}{x}^2 -|
\ip{x}{{x}^{\sharp}}|^2},$
\item the triple product satisfies $$\J{a}{b}{c}=\frac{1}{2}[\ip{a}{b}c+\ip{c}{b}a-\ip{a}{c^\sharp}b^\sharp],$$
\end{enumerate} is called a \emph{(complex) spin factor}.\smallskip

Throughout this section, $A$ will be a (complex) spin factor and  the duality of $A$ with $A^*$
will be denoted by $\pair{\cdot}{\cdot}$, while $J:A\rightarrow A^*$ will stand for the Riesz map.

The following  lemma shows that ternary derivations from $A$ to $A^*$ are in bijective correspondence
with the (linear) ternary derivations on $A$.

\begin{lemma}\label{l 3.6}%3.2.1 now 3.6
Let $D: A\to A$ be a linear mapping. Then denoting $\delta=J\circ D$, the identities
\begin{enumerate}[$(i)$]
\item $\pair{\J{\delta (a)}{b}{a}}{c}=\ip{c}{\J{D(a)}{b}{a}}$,
\item $\pair{\J{a}{\delta (b)}{a}}{c}=\ip{c}{\J{a}{D(b)}{a}}$,
\end{enumerate} hold for every $a,b,c$ in $A$. Consequently,
$D$ is a (linear) ternary derivation on $A$ if, and only if, $\delta$
lies in $\mathcal D_t(A,A^*).$
\end{lemma}

\begin{proof} To prove the first two statements, note that for  $a,b,c\in A$,
\begin{eqnarray*}
2
\pair{\Jc{\delta (a)}{b}{a}}{c}&=&2\pair{\delta (a)}{\J{b}{a}{c}}=2\ip{\J{b}{a}{c}}{D(a)}\\
&=&\ip{\ip{b}{a}c+\ip{c}{a}b-\ip{b}{c^\sharp}a^\sharp}{D(a)}\\
&=&\ip{\ip{b}{a}c}{D(a)}+\ip{\ip{c}{a}b}{D(a)}-\ip{\ip{b}{c^\sharp}a^\sharp}{D(a)}\\
&=&\ip{c}{\ip{a}{b}D(a)}+\ip{c}{\ip{D(a)}{b}a}-\ip{c}{\ip{D(a)}{a^\sharp}b^\sharp}\\
&=&\ip{c}{[\ip{a}{b}D(a)+\ip{D(a)}{b}a-\ip{D(a)}{a^\sharp}b^\sharp]}\\
&=&2\ip{c}{\Jc{D(a)}{b}{a}},
\end{eqnarray*}
and
\begin{eqnarray*}
\pair{\Jc{a}{\delta (b)}{a}}{c}&=&\overline{\pair{\delta (b)}{\J{a}{c}{a}}}=\ip{D(b)}{\J{a}{c}{a}}\\
&=&\ip{D(b)}{[ \ip{a}{c}a-\frac{1}{2} \ip{a}{a^\sharp}c^\sharp]}\\
&=&\ip{c}{[ \ip{a}{D(b)}a-\frac{1}{2} \ip{a}{a^\sharp} D(b)^\sharp]}
=\ip{c}{\Jc{a}{D(b)}{a}}.
\end{eqnarray*}

Finally, if $D$ is a (linear) ternary derivation on $A$ and $x\in A$, by $(i)$ and $(ii)$ we have
\begin{eqnarray*}
\pair{\delta\J{a}{b}{a}}{x}&=&\ip{x}{D\J{a}{b}{a}}\\
&=&\ip{x}{2\Jc{D(a)}{b}{a}}+\ip{x}{\Jc{a}{D(b)}{a}}\\
&=&\pair{2\Jc{\delta (a)}{b}{a}}{x}+\pair{\Jc{a}{\delta (b)}{a}}{x}
\end{eqnarray*}
so that $\delta\in \mathcal D_t(A,A^*)$. Similarly,
if $\delta\in \mathcal D_t(A,A^*)$, then $D$ is a (linear) ternary derivation on $A$.
\end{proof}

We deal now with inner ternary derivations.

\begin{lemma}\label{l 3.7}%3.2.3 now 3.7
For each element $a$ in $A$, let $\varphi=J(a)\in A^*$. Then for all $b,x,y\in A$, we have:
\begin{enumerate}[$(i)$]
\item $\pair{\J{b}{\varphi}{x}}{y}=\ip{y}{\J{b}{a}{x}},$
\item $\pair{\J{\varphi}{b}{x}}{y}=\ip{y}{\J{a}{b}{x}}$.
\end{enumerate} It follows that a linear mapping $D:A \to A$ is an inner ternary derivation if, and only if,
$\delta=J\circ D\in \mathcal{I}nn_t(A,A^*).$
\end{lemma}

\begin{proof}
The  first two statements follow by straightforward calculations.
If $D$ is an inner (linear) derivation on $A$ of the form $D= \delta (b,a)$ with $a,b\in A$,
then for $x,y\in A$,
\begin{eqnarray*}
\pair{\delta (x)}{y}&=&\ip{y}{D(x)}\\
&=&\ip{y}{\J{b}{a}{x}-\J{a}{b}{x}}\\
&=&\pair{\Jc{b}{J(a)}{x}}{y}-\pair{\Jc{J(a)}{b}{x}}{y}\\
&=&\pair{\delta(b,J(a)) (x)}{y}
\end{eqnarray*}
so that $\delta\in \mathcal Inn_t(A,A^*)$. Similarly, if $\delta\in \mathcal Inn_t(A,A^*)$,
then $D$ is an inner (linear) derivation on $A$.
\end{proof}

\begin{proposition}\label{p 3.8}%3.1.5 now 3.8
A spin factor is ternary weakly amenable if, and only if, it is finite dimensional.
\end{proposition}

\begin{proof}
Combining Lemmas \ref{l 3.6} and \ref{l 3.7}, $A$ is ternary weakly amenable if, and only if, $A$ has the inner derivation property.
Thus, applying \cite[Theorem 3]{HoMarPeRu} and the fact that every finite dimensional JB$^*$-triple
has the inner derivation property (cf. Proposition \ref{p finite dimensional}), we get the desired equivalence.
\end{proof}

The real forms of (complex) spin factors are called real spin factors. The next corollary is a direct consequence of Propositions \ref{p 3.8} and \ref{p complexification}.

\begin{corollary}\label{c 3.8 real}%3.1.5 now 3.8
A real spin factor is ternary weakly amenable if, and only if, it is finite dimensional.$\hfill\Box$
\end{corollary}

The following questions have been intractable up to this moment.

\begin{problem}
Are Cartan factors of type II and III ternary weakly amenable?
\end{problem}

\begin{problem}
Does there exist an infinite-rank rectangular Cartan factor of type I
which is ternary weakly amenable?
\end{problem}

\section{Commutative JB$^*$-triples are almost ternary weakly amenable}\bigskip

In this section we prove that every commutative real or complex
JB$^*$-triple is almost ternary weakly amenable. More concretely,
we prove that every ternary derivation from a commutative real or
complex JB$^*$-triple into its dual can be approximated in norm by
an inner derivation.\smallskip

We shall make use of the Gelfand representation theory for
commutative JB$^*$-triples (cf. \cite[\S 1]{Ka} and \cite{FriRus83}).
Let us denote $\mathbb{T} := \{ \alpha\in \mathbb{C}: |\alpha|=1\}$.
Given a commutative (complex) JB$^*$-triple $E$, there exists a principal $\mathbb{T}$-bundle $\Lambda =\Lambda(E)$, i.e.
a locally compact Hausdorff space $\Lambda$ together with a continuous mapping $\mathbb{T} \times \Lambda \to \Lambda$, $(t,\lambda) \mapsto t \lambda$ such that $s(t\lambda) = (s t ) \lambda$, $1 \lambda = \lambda$ and $t \lambda = \lambda \Rightarrow t=1$, satisfying that $E$ is JB$^*$-triple isomorphic to $$\mathcal{C}_{0}^{\mathbb{T}} (\Lambda) := \{f\in
\mathcal{C}_{0} (\Lambda) : f(t \lambda) = t f(\lambda),
\forall t\in \mathbb{T}, \lambda\in \Lambda \}. $$ 
We notice that $\mathcal{C}_{0}^{\mathbb{T}} (\Lambda) $ is a
JB$^*$-subtriple of the commutative C$^*$-algebra
$\mathcal{C}_{0} (\Lambda)$. Every commutative JB$^*$-triple is 
a $C_{_{\sum}}$-space and hence a complex Lindenstrauss space in the terminology of Olsen \cite{Ol74}.\smallskip

An element $e$ in a Jordan triple $E$ is called \emph{tripotent}
if $\J eee =e$. Each tripotent $e$ in $E$ induces a
decomposition of $E$ (called \emph{Peirce decomposition}) in the
form: $$E=E_0(e)\oplus E_1(e)\oplus E_2(e),$$ where
$E_k(e)=\{x\in E:L(e,e)x=\frac k2 x\}$ for $k=0,1,2$.\smallskip

The Peirce space $E_2 (e)$ is a unital JB$^*$-algebra with unit $e$,
product $x\circ_e y := \J xey$ and involution $x^{\sharp_e} := \J
exe$, respectively.\smallskip

A tripotent $e$ in $E$ is said to be \emph{unitary} if $L(e,e)$
coincides with the identity map on $E,$ equivalently, $E_2 (e) =
E$. When $E_0 (e) =\{0\}$, the tripotent $e$ is called
\emph{complete}.\smallskip

The proof given in Lemma \ref{l delta(1) anti-symmetric} remains
valid in the following setting:

\begin{lemma}
\label{l delta(1) anti-symmetric unitary} Let $E$ be a
JB$^*$-triple containing a unitary tripotent $u$. Every ternary
derivation $\delta $ in $\mathcal{D}_t (E,E^*)$ satisfies the
identity $\delta(u)^{\sharp_u} =- \delta(u)$, that is,
$\overline{\delta(u) (a^{\sharp_u})} =- \delta(u) (a) ,$ for every
$a$ in $E$.$\hfill\Box$
\end{lemma}

The following lemma summarises some basic properties of
commutative JB$^*$-triples, an implicit proof can be found combining
Theorems 2 and 4 in \cite{FriRus82}.

\begin{lemma}\label{l unital abelian JB*-triple}
Let $u$ be a norm-one element in a commutative JB$^*$-triple
$E\cong \mathcal{C}_{0}^{\mathbb{T}} (\Lambda (E))$. The following statements are equivalent:
\begin{enumerate}[$(a)$] \item $u$ is a complete tripotent;
\item $u$ is a unitary element;
\item $u$ is an extreme point of the unit ball of $E$.
\end{enumerate} If $u$ satisfies one of the above conditions
then $E$ is a commutative C$^*$-algebra with unit $u$, product and involution
given by $a\circ_u b := \J aub$ and $a^{\sharp_u} := \J
uau$ ($a,b\in E$), respectively.$\hfill\Box$
\end{lemma}

\begin{corollary}\label{l unital abelian JB*-triple b}
Every commutative JB$^*$-triple, $E,$ containing a complete
tripotent $u$ is ternary weakly amenable. Further, every ternary
derivation $\delta: E\to E^*$ can be written in the form $\delta =
-\frac{1}{2}\delta(u,\delta(u))= \J {\delta (u)}{u}{.}$.
\end{corollary}

\begin{proof} Lemma \ref{l unital abelian JB*-triple} shows
that $E$ is a commutative  C$^*$-algebra with product and
involution given by  $a\circ_u b := \J aub$ and $a^{\sharp_u} := \J uau$ ($a,b\in E$),
respectively. By the proof of  Lemma \ref{l 3.4.3} (see also Corollary \ref{c 3.4.6}) every ternary
derivation $\delta: E\to E^*$ may be written in the form $\delta =
-\frac{1}{2}\delta(u,\delta(u))$. Given $a,b$ in $E$, since $\{uba\}^{\sharp_u}=\{uab\}$,  we have
$$\delta (a) (b) =  -\frac{1}{2}\delta(u,\delta(u))  (a) (b) = -\frac{1}{2} \J {u}{\delta (u)}{a} (b)
+ \frac{1}{2} \J {\delta (u)}{u}{a} (b)$$ $$= \frac12 \left(
\delta(u) \J uab - \overline{\delta(u) \J uba} \right)= \frac12
\left( \delta(u) \J uab - {\delta(u)^{\sharp_u} \J uab} \right)
$$ $$=\hbox{ (by Lemma \ref{l delta(1) anti-symmetric unitary}) } \delta(u) \J uab
=\J {\delta (u)}{u}{a} (b),$$ which proves the last
identity.\end{proof}

\begin{corollary}\label{c commutative JBW*-triples} Every commutative
JBW$^*$-triple $E$ is (isometrically JB$^*$-triple isomorphic to)
a commutative von Neumann algebra, and thus it is ternary weakly
amenable. Moreover, every ternary derivation $\delta: E\to E^*$
can be written in the form $\delta =
-\frac{1}{2}\delta(u,\delta(u))= \J {\delta (u)}{u}{.}$, where $u$
is any complete tripotent in $E$.
\end{corollary}

\begin{proof} Let $E$ be a commutative JBW$^*$-triple.  The first assertion follows, indirectly, from
 \cite[Remark 2.7]{FriRus83} (see also \cite{FriRus82} or Lemma \ref{l unital abelian JB*-triple}). Since $E$ is a dual Banach space,
it follows from the Krein-Milman Theorem that the closed unit ball
of $E$, contains an extreme point. Lemma \ref{l unital abelian JB*-triple} above
assures that every extreme point of the unit ball of ${E}$ is a
complete tripotent in $E$. Therefore $E$ is a commutative
JB$^*$-triple containing a complete tripotent and the desired statement
follows from Corollary \ref{l unital abelian JB*-triple b}.
\end{proof}

\begin{corollary}\label{c commutative JB*-triples weakly amenable with unitary in bidual}
Let $E$ be a commutative JB$^*$-triple. Then every derivation
$\delta$ in $\mathcal{D}_t (E,E^*)$ may be written  in the form $\delta =
-\frac{1}{2}\delta(u,\delta^{**}(u))= \J {\delta^{**} (u)}{u}{.}$,
where $u$ is any complete (unitary) tripotent in $E^{**}$ and
$\delta^{**} (u)\in E^{*}$.
\end{corollary}

\begin{proof} Let  $\delta: E \to E^*$
be a ternary derivation. Since the triple product of $E^{**}$ is
separately weak$^*$-continuous, it follows from Goldstine's
Theorem that $E^{**}$ is a commutative JBW$^*$-triple. Proposition
\ref{p bidual extensions} guarantees that $\delta^{**} : E^{**}
\to E^{***}$ is a weak$^*$-continuous ternary derivation with
$\delta^{**} (E^{**}) \subseteq E^*$. Corollary \ref{c commutative
JBW*-triples} gives the desired statement.
\end{proof}

From now on, let $E$ be a commutative JB$^*$-triple which
is identified with $\mathcal{C}_{0}^{\mathbb{T}}
(\Lambda)$. For later use, we highlight the following properties: Let
$\mathcal{C}_{0}^{1} (\Lambda)$ denote the C$^*$-subalgebra of
$\mathcal{C}_{0} (\Lambda)$ of all $\mathbb{T}$-invariant
functions, that is,
$$\mathcal{C}_{0}^{1} (\Lambda) := \{f\in
\mathcal{C}_{0} (\Lambda) : f(t \lambda) = f(\lambda), \forall
t\in \mathbb{T}, \lambda\in \Lambda \}.$$

It is clear that for every $a,b$ in $E$ and $c$ in $\mathcal{C}_{0}^{1} (\Lambda)$,
the products $a b^*$ and $a c$ lie in $\mathcal{C}_{0}^{1} (\Lambda)$
and in $E$,  respectively.\smallskip

The mapping $$E\times E \to \mathcal{C}_{0}^{1} (\Lambda),$$
$$(a,b)\mapsto a b^*$$ is sesquilinear and positive ($a a^* \geq 0$
and $a a^* = 0\Longleftrightarrow a=0$). The products $$E\times
\mathcal{C}_{0}^{1} (\Lambda)\to E,$$ $$(a,c)\mapsto ac$$ and
$$E^*\times \mathcal{C}_{0}^{1} (\Lambda)\to E^*,$$ $$(\phi,c)
\mapsto (\phi c) (a) = \phi (ac)$$ equip $E$ and $E^*$ with a
structure of Banach $\mathcal{C}_{0}^{1} (\Lambda)$-bimodules.
We also have two mappings $E^*\times E\to \mathcal{C}_{0}^{1}
(\Lambda)^*$ and $\mathcal{C}_{0}^{1}
(\Lambda)^*\times E\to E^*$ defined by $\phi a (c) := \phi (ac)$
and $\psi a (b) = \psi (a^* b)$ ($\phi\in E^*$, $\psi \in \mathcal{C}_{0}^{1}
(\Lambda)^*$, $c\in \mathcal{C}_{0}^{1} (\Lambda)$, $a,b\in E$), respectively.\smallskip

We shall regard $E=\mathcal{C}_{0}^{\mathbb{T}}
(\Lambda)$ and $\mathcal{C}_{0}^{1} (\Lambda)$ as norm-closed JB$^*$-subtriples
of the C$^*$-algebra $A =\mathcal{C}_{0} (\Lambda)$. We shall identify
the weak$^*$ closure, in $A^{**}$, of a closed subspace
$Y$ of $A$ with $Y^{**}$. It follows from the separate weak$^*$-continuity
of the triple product in $A^{**}$, that for every $a,b$ in $E^{**}$ and $c$ in $\mathcal{C}_{0}^{1} (\Lambda)^{**}$, 
the products $a b^*$ and $a c$ lie in $\mathcal{C}_{0}^{1} (\Lambda)^{**}$
and in $E^{**}$,  respectively. Clearly the mappings defined in the previous paragraph extend
to $E^{**}\times E^{**}$ and $E^{**}\times
\mathcal{C}_{0}^{1} (\Lambda)^{**}$ and $E^*\times E^{**}$, respectively. \smallskip

One of the main consequences of the Gelfand theory for JB$^*$-triples provides
a structure theorem for the JB$^*$-subtriples generated by a single element.
Concretely speaking, for each element $a$ in a JB$^*$-triple $F$,
the JB$^*$-subtriple of $F$ generated
by the element $a$, $F_a$, is JB$^*$-triple isomorphic (and hence
isometric) to $C_0 (L)$ for some locally compact Hausdorff space
$L$ contained in $(0,\|a\|],$ such that $L\cup \{0\}$ is compact,
where $C_0 (L)$ denotes the Banach space of all complex-valued
continuous functions vanishing at $0.$ It is also known that there
exists a triple isomorphism  $\Psi$ from $F_a$ onto $C_{0}(L),$
such that $\Psi (a) (t) = t$ $(t\in L)$ (cf. \cite[Corollary
1.15]{Ka}). Consequently, for each element $a$ in $F$ there exists a (unique) element
$b\in F_a$ satisfying $\J bbb = a$. This element $b$ is usually called the
\emph{cube root} of $a$.\smallskip

The following proposition shows that, in the parlance of TROs (ternary rings of operators \cite{KauRua02}), 
$\mathcal{C}_{0}^{1} (\Lambda)$ identifies with the left and right linking $C^*$-algebras of the TRO $ \mathcal{C}_{0}^{\mathbb{T}}
(\Lambda)$. The proof is included here for completeness reasons.

\begin{proposition}\label{prop surjectivity} Let $E=\mathcal{C}_{0}^{\mathbb{T}}
(\Lambda)$ be a commutative JB$^*$-triple. The norm-closed
linear span of the set $$ \overline{E}\cdot E := \{ a^* b : a,b\in
E\}$$ coincides with $\mathcal{C}_{0}^{1} (\Lambda)$. %The same is true for set
%$E\cdot  \overline{E} := \{ a b^* : a,b\in E\}$.
\end{proposition}

\begin{proof} Let $B$ denote the norm-closed linear span of the set
$\overline{E}\cdot E.$ Given $a,b,c$ and $d$ in $E$
the product $bc^*d$ lies in $E$ and hence $(a^* b) (c^* d)= a^*(bc^*d)$ belongs to $\overline{E}\cdot E$.
Thus, $ \overline{E}\cdot E$ is multiplicatively closed and clearly self-adjoint (i.e. $\left(\overline{E}\cdot E\right)^*
=  \overline{E}\cdot E)$. We deduce that $B$ is a norm-closed *-subalgebra of $\mathcal{C}_{0}^{1} (\Lambda)$.\smallskip

We observe that $\mathcal{C}_{0}^{1} (\Lambda)$ is triple isometrically isomorphic to $\mathcal{C}_{0} (\Lambda/\mathbb{T})$
(via the canonical identification $c\mapsto \widehat{c}$, where $\widehat{c} (\lambda+\mathbb{T}) := c(\lambda),$ for every
$c\in  \mathcal{C}_{0}^{1} (\Lambda)$, $\lambda \in \Lambda$). We shall identify $\mathcal{C}_{0}^{1} (\Lambda)$ and
$\mathcal{C}_{0} (\Lambda/\mathbb{T})$. We claim that, under this identification, $B$ is a norm-closed *-subalgebra
of $\mathcal{C}_{0} (\Lambda/\mathbb{T})$, separates the points of $\Lambda/\mathbb{T}$ and vanishes nowhere.\smallskip

To this end, we claim first that $E$ separates the points of $\Lambda$ and vanishes nowhere, that is,
given $\lambda\in \Lambda$ there exists $a\in E$ 
with $a(\lambda)=1$. By Urysohn's lemma, there exists $f\in \mathcal{C}_{0} (\Lambda)$ satisfying $f(t\lambda ) =1,$ 
for every $t\in \mathbb{T}$. Let $d\mu$ denote the unit Haar measure on $\mathbb{T}$, the assignment 
$g\mapsto \pi (g) (\lambda) :=\int_{\mathbb{T}} t^{-1} g(t \lambda) d\mu(t)$ defines a contractive projection on $\mathcal{C}_{0} (\Lambda)$ 
whose image coincides with $E$. It is clear that $a=\pi (f)\in E$ and $a(\lambda)= \pi (f) (\lambda) = 1$. Take $\lambda_1\neq \lambda_2$ in $\Lambda$. We may assume that $\lambda_1 +\mathbb{T} \neq \lambda_2 +\mathbb{T}$, that is, the orbits of $\lambda_1$ and $\lambda_2$ are two compact disjoint subsets of $\Lambda$. Applying Urysohn's lemma, we find an element $f\in \mathcal{C}_{0} (\Lambda)$ satisfying $f(t\lambda_1 ) =1$ and $f(t\lambda_2 ) =0$,
for every $t\in \mathbb{T}$. The element $a= \pi (f)$ satisfies $a(\lambda_1)=1$ and $a(\lambda_2)=0.$\smallskip

Let us now  take $\lambda_1 + \mathbb{T}\neq \lambda_2 + \mathbb{T}$ in $\Lambda/\mathbb{T}$. Suppose that
$a^*b (\lambda_1) = a^*b (\lambda_2)$, for every $a,b\in E$. In particular, $a^*a (\lambda_1) = a^*a (\lambda_2)$,
and hence $aa^*a (\lambda_1) = aa^*a (\lambda_2)$, for every $a\in E$. Since every element $b$ in $E$
admits a cube root $a\in E_b$ satisfying $\J aaa =b$, we deduce that $b (\lambda_1) = b (\lambda_2)$, for every $b\in E$,
which is impossible because $\lambda_1 \neq \lambda_2$.
By the same argument, $B$ vanishes nowhere, so  the Stone-Weierstrass theorem assures
that $B= \mathcal{C}_{0}^{1} (\Lambda)$.
\end{proof}

\begin{theorem}\label{t commutative JB*-triples almost weakly amenable}
Every commutative (real or complex) JB$^*$-triple $E$ is almost ternary weakly
amenable, that is, $\mathcal{I}nn_{t} (E,E^*)$ is a norm-dense subset of
$\mathcal{D}_t (E,E^*)$.
\end{theorem}

\begin{proof} By Proposition \ref{p complexification}, we may assume that $E$
is a commutative complex JB$^*$-triple. We write $E =\mathcal{C}_0^{\mathbb{T}} (\Lambda (E))$
and $A =\mathcal{C}_{0} (\Lambda (E))$.
Let $\delta: E \to E^*$ be a ternary derivation. By
Corollary \ref{c commutative JB*-triples weakly amenable with
unitary in bidual}, $\delta^{**}=
-\frac{1}{2}\delta(u,\delta^{**}(u)) ,$ where $u$ is
a unitary in $E^{**}\subseteq A^{**}$ and $\psi=\delta^{**} (u) \in E^*$. In this case
$$\delta (a) (b) = -\frac12 \left(\overline{\psi (u b^* a)}- \psi (u a^* b)
\right),$$ for every $a,b\in E$, where the products are taken in
$\mathcal{C}_{0} (\Lambda (E))^{**}$.\smallskip

The mapping $c\mapsto \psi (u c)$ defines a functional in the dual
of $\mathcal{C}_{0}^{1} (\Lambda (E))$. Since the latter is a
C$^*$-algebra, by Cohen's factorisation theorem (cf. \cite[Theorem
VIII.32.22]{HewRoss}), there exist $\varphi\in \mathcal{C}_{0}^{1}
(\Lambda (E))^*$ and $d\in \mathcal{C}_{0}^{1} (\Lambda (E))$ such
that $\psi (u c) = \varphi (d c),$ for every $c\in
\mathcal{C}_{0}^{1} (\Lambda (E))$. Therefore, for each $a,b\in E$ we have
$$\delta  (a) (b) =
-\frac12 \left(\overline{\psi (u b^* a)}- \psi (u a^* b) \right) =
-\frac12 \left(\overline{\varphi (d b^* a)}- \varphi (d a^* b)
\right).$$\smallskip

 Given $\varepsilon >0$, by Proposition \ref{prop surjectivity}
there exist $x_1,y_1, \ldots, x_n,y_n\in E$ satisfying $\| d -\sum_{j=1}^{n} y_j^* x_j \|
<\varepsilon$. Let $\phi_j = \varphi y_j^*\in E^*$ ($j=1,\ldots , n$).
The sum $-\frac12 \sum_{j=1}^{n} \delta(x_j,\phi_j)$ defines a inner ternary
derivation from $E$ to $E^*$. Given $a,b\in E$, we have:
$$\left|\delta (a)(b) +\frac{1}{2} \sum_{j=1}^{n} \delta(x_j,\phi_j) (a) (b) \right| $$ $$= \left| -\frac12 \left(\overline{\varphi (d b^* a)}- \varphi (d a^* b)
\right) +\frac12 \sum_{j=1}^{n}  \left(\overline{\phi_j (x_j b^* a)}- \phi_j (x_j a^* b) \right)\right|$$
$$= \left| -\frac12 \left(\overline{\varphi (d b^* a)}- \varphi (d a^* b)
\right) +\frac12 \sum_{j=1}^{n}  \left(\overline{\varphi (y_j^* x_j b^* a)}- \varphi (y_j^* x_j a^* b) \right)\right|$$
$$= \left| \frac12 \ \overline{\varphi \left( \left(\sum_{j=1}^{n} y_j^* x_j-d \right) b^* a\right)} + \frac12 \ {\varphi \left( \left(d-\sum_{j=1}^{n}  y_j^* x_j \right) a^* b \right)}\right|$$ $$\leq\frac{1}{2} \|\varphi\| \ \|a\| \ \|b\| \ \left\| d -\sum_{j=1}^{n} y_j^* x_j \right\| < \varepsilon\|\varphi\| \ \|a\| \ \|b\|/2.$$
Thus,  $\left\|\delta-(-\frac{1}{2} \sum_{j=1}^{n} \delta(x_j,\phi_j))\right\|< \varepsilon\|\varphi\|/2.$
\end{proof}

The proof given in the above theorem shows that under additional
hypothesis on the set $E\cdot \overline{E} := \{ a b^* : a,b\in
E\}$ a commutative JB$^*$-triple $E$ is ternary weakly amenable.

\begin{corollary}\label{c commutative JB*-triples weakly amenable + surjectivity}
Let $E=\mathcal{C}_{0}^{\mathbb{T}} (\Lambda (E))$ be a
commutative JB$^*$-triple. Suppose that the linear span of the set
$E\cdot \overline{E} := \{ a b^* : a,b\in E\}$ coincides with
$\mathcal{C}_{0}^{1} (\Lambda (E))$. Then $E$ is ternary weakly
amenable.$\hfill\Box$
\end{corollary}

The question clearly is whether the additional hypothesis in
Corollary \ref{c commutative JB*-triples weakly amenable +
surjectivity} is automatically satisfied for every commutative
JB$^*$-triple $E$. We do not know the answer, the best result we
could obtain in this line is Proposition \ref{prop surjectivity}.\smallskip

Related to this topic, we can say that given a commutative
JB$^*$-triple $E=\mathcal{C}_{0}^{\mathbb{T}} (\Lambda (E))$, the
mapping $E\times E \to \mathcal{C}_{0}^{1} (\Lambda (E)),$
$(a,b)\mapsto a b^*$ need not be, in general, surjective.
Indeed, let $L$ be a locally compact Hausdorff space. We shall say that
$L$ is a locally compact \emph{principal $\mathbb{T}$-bundle} if
there exists a continuous mapping $\mathbb{T}\times L \to L$,
$(t,\lambda)\mapsto t  \lambda$ satisfying $s(t \lambda) = (st)
\lambda$ and $1 \lambda = \lambda$, for every $s,t\in \mathbb{T}$,
$\lambda\in L$. We write $$\mathcal{C}_{0}^{\mathbb{T}} (L) :=\{
f\in \mathcal{C}_{0} (L) : f(t \ \lambda) = t \ f(\lambda)
\ (t\in\mathbb{T}, \lambda\in L)\}.$$ It
is known that $\mathcal{C}_{0}^{\mathbb{T}} (L)$ is isometrically
isomorphic to to $C_0 (L^{\prime})$ for some locally compact
$L^{\prime}$ if and only if $L$ is a trivial $\mathbb{T}$-bundle,
i.e. $L/\mathbb{T} \times \mathbb{T} \cong L$ (cf. \cite[Corollary
1.13]{Ka}). The set $S:= \{ z\in \mathbb{C}^{n+1} : \|z\|_2= 1\}$
is compact and a non-trivial principal $\mathbb{T}$-bundle. Let
$E= \mathcal{C}^{\mathbb{T}}(S)\subset \mathcal{C}(S)$ and
$\mathcal{C}^{1} (S) :=\{f\in\mathcal{C} (S) : f(t \ z) = f(z) \ (t\in\mathbb{T}, z\in S)\}.$
We can obviously identify $S$ with a closed subset of $\Lambda (E)$ which satisfies $\mathbb{T} S =S$.\smallskip

If the mapping $E\times E \to \mathcal{C}_{0}^{1} (\Lambda (E)),$
$(a,b)\mapsto a b^*$ were surjective, then, applying Urysohn's lemma,
there would exist functions $a,b\in E$ satisfying $a b^* = v$,
where ${v}\in \mathcal{C}_{0}^{1} (\Lambda (E))\cong \mathcal{C}_{0} (\Lambda (E)/\mathbb{T})$
is a function satisfying $v (z)=1$, for every $z\in S$.
In this case the function $z\mapsto u(z) = \frac{a(z)}{|a(z)|}$ $(z\in S),$
would be a unitary element in $E$, and hence, by Lemma \ref{l unital abelian
JB*-triple}, $E$ would be an abelian C$^*$-algebra, which is
impossible because $S$ is a non-trivial principal $\mathbb{T}$-bundle.

\bigskip


\smallskip


\begin{thebibliography}{10}

\bibitem{AlBreVill} J. Alaminos, M. Bre\v{s}ar, A.R. Villena, The strong degree of von Neumann algebras
and the structure of Lie and Jordan derivations, \emph{Math. Proc. Cambridge Philos. Soc.} \textbf{137},
no. 2, 441-463 (2004).

\bibitem{AlbMuc57} A.A. Albert, B. Muckenhoupt, On matrices of trace zero, \emph{Michigan Math. J.} \textbf{4},
1-3 (1957).

%\bibitem{BaCur} W.G. Bade,  P.C. Curtis, Homomorphisms of commutative Banach algebras,
%\emph{Amer. J. Math.} \textbf{82} 589-608 (1960).

\bibitem{BarFri} T.J. Barton, Y. Friedman, Bounded derivations of JB$^*$-triples,
\emph{Quart. J. Math. Oxford} \textbf{41}, 255-268 (1990).

\bibitem{BarTi} T.J. Barton, R.M. Timoney, Weak$^*$-continuity of Jordan
triple products and its applications. Math. Scand. {\bf 59},
177-191 (1986).

%\bibitem{BeLoPeRo} J. Becerra-Guerrero, G. L\'{o}pez, A. M. Peralta and
%A. Rodr\'{\i}guez-Palacios, Relatively weakly open sets in closed
%unit balls of Banach spaces, and real JB$^*$-triples of finite
%rank, \emph{Math. Ann.} \textbf{330}, 45-58 (2004).

%\bibitem{BuChu2} L.J. Bunce, C.H. Chu, Dual spaces of JB$^*$-triples
%and the Radon-Nikodym property, \emph{Math. Z.} \textbf{208}, 327-334
%(1991).

\bibitem{BuChu} L.J. Bunce, C.H. Chu, {Compact operations,
multipliers and the Radon-Nikodym property in JB$^*$-triples},
\emph{Pacific J. Math.} \textbf{135}, 249-265 (1992).

%\bibitem{Bun86} L. J. Bunce, Structure of representations and ideals of
%homogeneous type in Jordan algebras, \emph{Quart. J. Math. Oxford
%Ser.} (2) \textbf{37}, no. 145, 1-10  (1986).

%\bibitem{BuChu} L.J. Bunce and C.-H. Chu,  Compact  operations, multipliers  and
%Radon-Nikodym property in $JB^*$-triples, {\it Pacific J. Math.}
%{\bf 153}, 249-265 (1992).

%\bibitem{BuChuZa1} L.J. Bunce, Ch.-H. Chu, B. Zalar, Structure spaces and decomposition in JB$^*$-triples.
%Math. Scand. {\bf 86} (2000), 17-35.

%\bibitem{BuFerMarPe} L.J. Bunce, F.J. Fern\'andez-Polo, J.M. Moreno, A.M. Peralta,
%A Sait\^o-Tomita-Lusin theorem for JB$^*$-triples and applications,
%Quart. J. Math. Oxford, {\bf 57} (2006), 37-48.

\bibitem{BunPasch} J.W. Bunce, W.L. Paschke, Derivations on a C$^*$-algebra and its
double dual, \emph{J. Funct. Analysis} \textbf{37}, 235-247
(1980).

%\bibitem{BurFerGarMarPe} M. Burgos, F.J. Fern{\'a}ndez-Polo, J. Garc{\'e}s, J.
%Mart{\'\i}nez, A.M. Peralta, Orthogonality preservers in
%C$^*$-algebras, JB$^*$-algebras and JB$^*$-triples, \emph{J. Math. Anal.
%Appl.} \textbf{348}, 220-233 (2008).

\bibitem{BurPeRaRu} M. Burgos, A.M. Peralta, M. Ram\'{\i}rez, M.E.
Ruiz Morillas, von Neumann regularity in Jordan-Banach triples,
to appear in \emph{Proceedings of Jordan structures in Algebra and
Analysis Meeting. Tribute to El Amin Kaidi for his 60th birthday.
Almer{\'\i}a, 20, 21 y 22 de Mayo de 2009}, Edited by J. Carmona et
al., Universidad de Almer{\'\i}a, 2010.

%\bibitem{CLev} S. B. Cleveland, Homomorphisms of non-commutative *-algebras,
%\emph{Pacific J. Math.} \textbf{13}, 1097-1109 (1963).

%\bibitem{Coh} P.J. Cohen, Factorization in group algebras, \emph{Duke Math. J.}
%\textbf{26}, 199-205 (1959).

%\bibitem{Cuntz} J. Cuntz, On the continuity of semi-norms on operator algebras, \emph{Math. Ann.}
%\textbf{220}, no. 2, 171-183 (1976).

%\bibitem{Dales78} H.G. Dales, Automatic continuity:
%a survey, \emph{Bull. London Math. Soc.} \textbf{10}, no. 2, 129-183 (1978).

%\bibitem{Dales89} H.G. Dales, On norms on algebras,
%\emph{Proc. Centre Math. Anal. Austral. Nat. Univ.} \textbf{21}, 61-69 (1989).

%\bibitem{Dales00} H.\ G.\ Dales, Banach algebras and automatic
%continuity. London Mathematical Society Monographs.
%New Series, 24. Oxford Science Publications. The Clarendon Press, Oxford University Press, New York, 2000.

\bibitem{connes} A. Connes, Classification of injective factors,
cases $II_1$, $II_\infty$, $III_\lambda$, $\lambda\ne 1$, \emph{Ann. of Math.}, \textbf{104}, 73-115 (1976).

\bibitem{Cusack75} J.M. Cusack, Jordan derivations on rings,
\emph{Proc. Amer. Math. Soc.} \textbf{53} no. 2, 321-324 (1975)

%\bibitem{DanFri87} T. Dang, Y. Friedman, Classification of JB$^*$-triple factors and applications
%\emph{Math. Scand.} \textbf{61}, 292-330 (1987).

%\bibitem{DaRu} T. Dang, B. Russo, Real Banach Jordan triples,
%\emph{Proc. Amer. Math. Soc.} \textbf{122}, 135-145 (1994).

\bibitem{Di86b} S. Dineen, The second dual of a JB$^*$-triple system,
in: J. Mujica (Ed.), Complex analysis, Functional Analysis and
Approximation Theory, North-Holland, Amsterdam, 1986.

%\bibitem{EdRu01} C.M. Edwards, G.T. R{\"u}ttimann, Orthogonal faces of the unit ball in a Banach space,
%\emph{Atti Sem. Mat. Fis. Univ. Modena} \textbf{49}, 473-493 (2001).

%\bibitem{FerGarPe} F.J. Fern{\' a}ndez-Polo, J. Garc{\'e}s, A.M. Peralta, A Kaplansky theorem for JB$^*$-triples,
%preprint 2010.

%\bibitem{FerMarPe} F.J. Fern{\'a}ndez-Polo, J. Mart{\' i}nez Moreno, and A.M. Peralta, Surjective isometries
%between real JB$^*$-triples, \emph{Math. Proc. Cambridge Phil. Soc.}, \textbf{137} 709-723 (2004).

\bibitem{FriRus82} Y. Friedman, B. Russo, Contractive projections on $C_{0}(K)$, \emph{Trans. Amer. Math. Soc.}
\textbf{273}, no. 1, 57-73 (1982).

\bibitem{FriRus83} Y. Friedman, B. Russo, Function representation of commutative operator triple systems, \emph{J. London Math. Soc.} (2) {\bf 27}, no. 3, 513-524 (1983).

\bibitem{FriRu85} Y. Friedman, B. Russo, Structure of the predual of a
JBW$^*$-triple, \emph{J. Reine Angew. Math.} {\bf 356}, 67-89 (1985).

\bibitem{FriRus86bis} Y. Friedman, B. Russo, A Gelfand-Naimark
theorem for JB$^*$-triples, \emph{Duke Math. J.}, {\bf 53}, 139-148 (1986).

%\bibitem{FriRu4} Y. Friedman and B. Russo, Solution of the contractive
%projection problem, \textit{J. Funct. Analysis}, \textbf{60}
%(1986), 56-79.

\bibitem{Fack} T. Fack, Finite sums of commutators in C$^*$-algebras,
\emph{Ann. Inst. Fourier (Grenoble)} \textbf{32}, 129-137 (1982).

\bibitem{Haa83} U. Haagerup, All nuclear C$^*$-algebras are amenable, \emph{Invent. Math.}
\textbf{74}, no. 2, 305-319 (1983).

\bibitem{HaaLaust} U. Haagerup, N.J. Laustsen, Weak amenability of C$^*$-algebras and a theorem of Goldstein,
in  \emph{Banach algebras '97 (Blaubeuren)}, 223-243, de Gruyter,
Berlin, 1998.

%\bibitem{Hanche} H. Hanche-Olsen, E. St{\o }rmer, \textit{Jordan operator algebras, Monographs and Studies in Mathematics 21}, Pitman, London-Boston-Melbourne %1984.

\bibitem{HejNik96} S. Hejazian, A. Niknam, Modules Annihilators and module derivations of JB$^*$-algebras,
\emph{Indian J. pure appl. Math.}, \textbf{27}, 129-140 (1996).

%\bibitem{HejNik} S. Hejazian, A. Niknam, A Kaplansky theorem for ${\rm
%JB}^*$-algebras, \emph{Rocky Mountain J. Math.} \textbf{28}, no.
%3, 977-982 (1998).

%\bibitem{Her} I.N. Herstein, Jordan derivations of prime rings,
%\emph{Proc. Amer. Math. Soc.}, \textbf{8}, 1104--1110 (1957).

\bibitem{HewRoss} E. Hewitt, K.A. Ross, \emph{Abstract harmonic analysis.
Vol. II: Structure and analysis for compact groups. Analysis on
locally compact Abelian groups}, Springer-Verlag, New York-Berlin,
1970.

\bibitem{HiPhi} E. Hille, R.S. Phillips, \emph{Functional
analysis and semigroups}, A.M.S. Colloquium Publications Vol. XXXI,
1957.

%\bibitem{Ho92} T. Ho, Derivations of Jordan Banach triples, Dissertation, University of California, Irvine 1992.

\bibitem{HoMarPeRu} T. Ho, J. Martinez-Moreno, A.M. Peralta, B. Russo,
 Derivations on real and complex JB$^\ast$-triples, \emph{J. London Math. Soc.} (2)
 \textbf{65}, no. 1, 85-102 (2002).

%\bibitem{Ho87} G. Horn, Characterization of the predual and ideal structure
%of a ${\rm JBW}^*$-triple, \emph{Math. Scand.} \textbf{61}, no. 1,
%117-133 (1987).

\bibitem{Hochs} G. Hochschild, Semi-Simple Algebras and Generalized Derivations,
\emph{Amer. J. Math.} \textbf{64}, No. 1, 677-694 (1942).

\bibitem{IsKaRo95} J.M. Isidro, W. Kaup, A. Rodr{\'\i}guez, On
real forms of JB$^*$-triples, \emph{Manuscripta Math.}
\textbf{86}, 311-335 (1995).

\bibitem{Jacobson51} N. Jacobson, General representation
theory of Jordan algebras, \emph{Trans. Amer. Math. Soc.} \textbf{70} 509-530 (1951)

\bibitem{Jac} N. Jacobson, \emph{Structure and representation of Jordan algebras},
Amer. Math. Soc. Colloq. Publications, vol 39, 1968.

\bibitem{John96} B.E. Johnson, Symmetric amenability and the nonexistence of Lie and
Jordan derivations, \emph{Math. Proc. Cambridge Philos. Soc.}
\textbf{120}, no. 3, 455-473 (1996).

%\bibitem{JohnSin68} B.E. Johnson, A.M. Sinclair, Continuity of
%derivations and a problem of Kaplansky, \emph{Amer. J. Math.} \textbf{90}, 1067-1073 (1968).

%\bibitem{JohnSin69} B.E. Johnson, A.M. Sinclair, Continuity of
%linear operators commuting with continuous linear operators. II.,
% \emph{Trans. Amer. Math. Soc.} \textbf{146} 533-540 (1969).

\bibitem{Kad} R. Kadison, Derivations of operator algebras, \emph{Ann. of Math.}, \textbf{83}, 280-293 (1966).

\bibitem{Kaplansky58} I. Kaplansky, \emph{Some aspects of analysis and probability, pp. 1�34}; in Surveys in Applied Mathematics. Vol. 4 John Wiley \& Sons, Inc., New York; Chapman \& Hall, Ltd., London 1958

%\bibitem{Kaplansky1949} I. Kaplansky, Normed algebras,
%\emph{Duke Math. J.} \textbf{16}, 399-418 (1949).

%\bibitem{Kaps} I. Kaplansky, Normed algebras,
%\emph{Duke Math. J.} \textbf{16}, 399-418 (1949).

\bibitem{Ka} W. Kaup, A Riemann Mapping Theorem for bounded symmentric domains in
complex Banach spaces, \emph{Math. Z.} \textbf{183}, 503-529
(1983).

\bibitem{Ka97} W. Kaup, On real Cartan factors, \textit{Manuscripta Math.} \textbf{92}, 191-222 (1997).

%\bibitem{Ka2} W. Kaup, Contractive projections on Jordan
%C*-algebras and generalizations, \textit{Math. Scand.}
%\textbf{54} (1984), 95-100.

%\bibitem{Ka96} W. Kaup, On spectral and
%singular values in JB$^*$-triples, \emph{Proc. Roy. Irish Acad. Sect.
%A} \textbf{96}, no. 1, 95-103 (1996).

%\bibitem{Ka97} W. Kaup, On real Cartan factors, \textit{Manuscripta
%Math.} \textbf{92}, 191-222 (1997).

%\bibitem{KaUp} W. Kaup, H. Upmeier, Jordan algebras and
%symmetric Siegel domains in Banach spaces,
%\emph{Math. Z.} \textbf{157}, 179-200 (1977).

%\bibitem{LinMat} Y.-F. Lin and M. Mathieu, Jordan isomorphism of purely infinite C$\sp *$-algebras,
% \emph{Quart. J. Math.}, \textbf{58}, 249-253 (2007).

\bibitem{KauRua02}
M. Kaur, Z-J Ruan, Local properties of ternary rings of operators and their linking C*-algebras, J. Funct. Anal. 195 (2002), no. 2, 262--305.

\bibitem{KR78} O. K\"uhn, A. Rosendahl,
Wedderburnzerlegung f\"ur Jordan-Paare, \emph{Manus. Math.} \textbf{24}, 403-435 (1978).

%\bibitem{Loos75} O. Loos, \emph{Jordan pairs},
%Lecture Notes in Mathematics,Vol. 460.
%Springer-Verlag, Berlin-New York, 1975.

\bibitem{Loos77} O. Loos, \emph{Bounded symmetric domains and Jordan pairs},
Math. Lectures, University of California, Irvine 1977.

%\bibitem{Lopezetal01} A.\  Fern{\' a}ndez Lopez, H.\  Marhnine, C.\  Zarhouti,
%Derivations on Banach-Jordan pairs. \emph{Q. J. Math.} \textbf{52}, no. 3, 269--283 (2001).

\bibitem{Marc02} L.W. Marcoux, On the linear span of the projections in certain simple C$^*$-algebras,
\emph{Indiana Univ. Math. J.} \textbf{51}, 753-771 (2002).

\bibitem{Marc06} L.W. Marcoux, Sums of small numbers of commutators, \emph{J. Oper. Theory} \textbf{56}:1, 111-142
(2006).

\bibitem{Marc10} L.W. Marcoux, Projections, commutators and Lie ideals in
C$^*$-algebras, \emph{Math. Proc. R. Ir. Acad.} \textbf{110} A, no. 1, 31-55
(2010).

\bibitem{MarPe} J. Mart{\'i}nez, A.M. Peralta, Separate weak*-continuity of the triple product
in dual real JB$^*$-triples, \emph{Math. Z.} \textbf{234}, 635-646
(2000).

\bibitem{Meyberg72} K. Meyberg, \emph{Lecture Notes on triple systems},
University of Virginia 1972.

\bibitem{Ol74} G.H. Olsen, On the classification of complex Lindenstrauss spaces, \emph{Math. Scand.}
\textbf{35}, 237-258 (1974).

\bibitem{PeaTop71} C. Pearcy, D. Topping, On commutators in ideals of compact operators, \emph{Mich. Math. J.} \textbf{18}, 247-252
(1971).

\bibitem{Ped} G.K. Pedersen, \emph{C$^*$-algebras and their automorphism groups},
Academic Press, London 1979.

\bibitem{PeRo} A.M. Peralta, A. Rodr{\'\i}guez Palacios,
Grothendieck's inequalities for real and complex JBW*-triples,
\emph{Proc. London Math. Soc.} (3) \textbf{83}, no. 3, 605-625
(2001).

\bibitem{PeRu} A.M. Peralta, B. Russo, Automatic continuity of derivations on
C$^*$-algebras and JB$^*$-triples, preprint 2010.

%\bibitem{Pal} T.W. Palmer, \emph{Banach algebras and the general theory of *-algebras},
%Vol. I., \emph{Algebras and Banach algebras}, Cambridge University Press, Cambridge, 1994.

%\bibitem{PeSta} A.M. Peralta, L.L. Stach{\'o}, Atomic decomposition of real ${\rm JBW}^*$-triples,
%\emph{Quart. J. Math. Oxford} \textbf{52}, no. 1, 79-87 (2001).

%\bibitem{Pie} A. Pietsch, {\it Operator Ideals,} North Holland, Amsterdam, 1980.

\bibitem{Pop} C. Pop, Finite sums of commutators, \emph{Proc. Amer. Math. Soc.} \textbf{130},
no. 10, 3039-3041 (2002).

%\bibitem{Rick50} C. Rickart, The uniqueness of norm problem in Banach algebras,
%\emph{Ann. of Math}, \textbf{51}, 615-628 (1950).

%\bibitem{Rick} C. Rickart, \emph{Genenal Theory of Banach Algebras}, Van Nostrand, New York, 1960.

\bibitem{Ringrose72} J.R. Ringrose, {Automatic continuity of derivations of operator algebras},
\emph{J. London Math. Soc.} (2) \textbf{5} , 432-438 (1972).

%\bibitem{Ringrose74} J.R. Ringrose, Linear functionals on operator algebras and their Abelian subalgebras,
%\emph{J. London Math. Soc.} (2) \textbf{7}, 553-560 (1974).

\bibitem{Russo94} B. Russo, {\it Structure of JB$^*$-triples},
In: Jordan Algebras, Proceedings of the Oberwolfach Conference
1992, Eds: W.\ Kaup,K.\ McCrimmon,H.\ Petersson, de Gruyter,
Berlin, 209-280 (1994).

\bibitem{Sak60} S. Sakai, On a conjecture of Kaplansky, \emph{Tohoku Math. J.}, \textbf{12} (1960), 31-33.

\bibitem{sakai} S. Sakai, Derivations of W$^*$-algebras, \emph{Ann. of Math.}, \textbf{83}, 273-279 (1966).

%\bibitem{Sa} S. Sakai, \emph{$C\sp*$-algebras and $W\sp*$-algebras}, in:
%Ergebnisse der Mathematik und ihrer Grenzgebiete, Band 60.
%Springer-Verlag, New York-Heidelberg. 1971.

\bibitem{Shoda37} K. Shoda, Einige S\"atze \"uber Matrizen,
\emph{Japanese J. Math} \textbf{13}, 361-265 (1937).

\bibitem{Sim} B.~Simon,
\emph{Trace ideals and their applications. Second edition},
American Mathematical Society, Providence, RI, 2005.

%\bibitem{Sin75} A.M. Sinclair, Homomorphisms from C$^{*} $-algebras,
%\emph{Proc. London Math. Soc.} (3) \textbf{29}, 435-452 (1975).

%\bibitem{Sin76} A.M. Sinclair, \emph{Automatic continuity of linear operators},
%London Mathematical Society Lecture Note Series, No. 21,
%Cambridge University Press, Cambridge-New York-Melbourne, 1976.

%\bibitem{Sta} L.L. Stacho, A projection principle concerning
%biholomorphic automorphisms, \emph{Acta Sci. Math.} \textbf{44} (1982), 99-124.

\bibitem{Sinclair70} A.M. Sinclair, Jordan homomorphisms and derivations on semisimple Banach algebras,
\emph{Proc. Amer. Math. Soc.} (3) \textbf{24}, 209-214 (1970).

\bibitem{Tak} M. Takesaki; {\em Theory of operator algebras. I.},
Springer-Verlag, New York-Heidelberg, 1979.

\bibitem{upmeier} H. Upmeier, Derivations of Jordan C$^*$-algebras, \emph{Math. Scand.}
 \textbf{46}, 251-264 (1980).

\bibitem{U1} H. Upmeier, \emph{Symmetric Banach manifolds and Jordan C*-algebras},
North-Holland, Amsterdam, 1985

\bibitem{U2} H. Upmeier, Jordan algebras in analysis, operator theory,
and quantum mechanics, CBMS, Regional conference, No. 67 (1987).

\bibitem{Villena96} A.R. Villena, Derivations on Jordan-Banach algebras. \emph{Studia Math.} \textbf{118}, no. 3, 20--229 (1996)

%\bibitem{Vill} A. R. Villena, Automatic continuity in associative and nonassociative context,
%\emph{Irish Math. Soc. Bull.} No. \textbf{46}, 43-76 (2001).

\bibitem{Weiss04} G. Weiss, {\it $B(H)$-commutators: a historical survey}, in: Operator Theory:
Advances and Applications, Vol. 153, 307-320,
 Birkh{\"u}auser Verlag, Basel, Switzerland, 2004.

\bibitem{Wright77} J.D. Maitland Wright,
Jordan C$^*$-algebras.
Michigan Math. J. 24, no. 3, 291-302 (1977).

%\bibitem{Yood} B. Yood, Homomorphisms on normed algebras, \emph{Pacific J. Math.}
%\textbf{8}, 373-381 (1958).

\end{thebibliography}
\end{document}